\documentclass[11pt]{amsart}
\usepackage{mathrsfs} 
\usepackage[margin=1.25in]{geometry}
\usepackage{ifthen}
\usepackage{graphicx}
\usepackage{epsfig}
\usepackage{dsfont, amssymb}
\usepackage{amsfonts,dsfont,amssymb,amsthm,stmaryrd,bbm}
\usepackage{amsmath}
\usepackage{xcolor}

\usepackage[toc,page]{appendix} 
\usepackage{hyperref,mathtools}
\hypersetup{colorlinks=true,
}

\let\originallesssim\lesssim
\let\originalgtrsim\gtrsim

\DeclareRobustCommand{\lesssim}{%
  \mathrel{\mathpalette\lowersim\originallesssim}%
}
\DeclareRobustCommand{\gtrsim}{%
  \mathrel{\mathpalette\lowersim\originalgtrsim}%
}

\makeatletter
\newcommand{\lowersim}[2]{%
  \sbox\z@{$#1<$}%
  \raisebox{-\dimexpr\height-\ht\z@}{$\m@th#1#2$}%
}
\makeatother

\newtheorem{thm}{Theorem}[section]

\newtheorem{remark}[thm]{Remark}
\newtheorem{lem}[thm]{Lemma}

\newtheorem{defn}[thm]{Definition}
\newtheorem{cor}[thm]{Corollary}

\newcommand\independent{\protect\mathpalette{\protect\independent}{\perp}} 
\def\independent#1#2{\mathrel{\rlap{$#1#2$}\mkern2mu{#1#2}}} 

\DeclareMathOperator*{\supess}{sup\,ess}
\DeclareMathOperator*{\infess}{inf\,ess}

\DeclareMathOperator{\Erf}{Erf}
\DeclareMathOperator{\Erfc}{Erfc}

\def\phi{\varphi}

\def\bee{\begin{eqnarray*}}
\def\ene{\end{eqnarray*}}
\usepackage{mathtools}

\begin{document}
\title[Transport-majorization]{Transport-majorization to analytic and geometric inequalities}

%opening
% \title[Nazarov-Podkoytov's Lemma, transport and majorization]{Nazarov-Podkoytov's Lemma, Ball's integral inequality and transport: A majorization perspective}

\author{James Melbourne and Cyril Roberto}

\thanks{The last author is supported by the Labex MME-DII funded by ANR, reference ANR-11-LBX-0023-01 and ANR-15-CE40-0020-03 - LSD - Large Stochastic Dynamics, and the grant of the Simone and Cino Del Duca Foundation, France.}

\address{Centro de Investigaci\'on en Matem\'aticas, Probabilidad y Estad\'isticas.: 36023 Guanajuato, Gto, Mexico.}
\address{Universit\'e Paris Nanterre, Modal'X, UMR 9023, FP2M, CNRS FR 2036, 200 avenue de la R\'epublique 92000 Nanterre, France.}

\email{james.melbourne@cimat.mx, croberto@math.cnrs.fr}

\date{\today}

\maketitle
 
\begin{abstract}
We introduce a transport-majorization argument that establishes a majorization in the convex order between two densities, based on control of the gradient of a transportation map between them.  As applications, we give elementary derivations of some delicate Fourier analytic inequalities, which in turn yield geometric ``slicing-inequalities'' in both continuous and discrete settings.
As a further consequence of our investigation we prove  that any strongly log-concave probability density  majorizes the Gaussian density and thus the Gaussian density maximizes the R\'enyi and Tsallis entropies of all orders among all strongly log-concave densities.
\end{abstract}

\section{Introduction}

Let us introduce the notion of Majorization which will play a key role in the investigations of this paper.
{
\begin{defn}[Marjorization] \label{defn: Majorization}
    For a finite signed measure $\sigma$ on a vector space $E$, we write $0 \prec \sigma$ and say that \emph{$\sigma$ majorizes $0$ in the convex order} when $\phi$ convex implies
    \[
        0 \leq \int_E \phi \ d \sigma.
    \]
\end{defn}
}
% \begin{defn}[Majorization]
% For $f$ and $g$ be non-negative functions such that $\int_\mathcal{X} f d\mu = \int_\mathcal{X} g d\mu < \infty$, we write {$g \prec f$ and say $f$ \emph{majorizes} $g$ when $f$ majorizes $g$ in the convex order.  That is when $t > 0$ implies
% \begin{align}
%     \int_t^\infty G(\lambda) d \lambda \leq \int_t^\infty F(\lambda) d \lambda.
% \end{align}}
% % When $X$ and $Y$ are $\mathcal{X}$ valued random variables with respective densities $f$ and $g$ with respect to $\mu$, we write $X \prec Y$ when $f \prec g$.
% \end{defn}

We will be particularly interested in using the machinery of majorization to make statements about the behavior of density functions.  This will correspond to the case that $\sigma$ is supported on $[0,\infty)$ and  is the difference of two postive measures, %$\sigma_1$ and $\sigma_2$, and that each $\sigma_i$ is
each a pushforward of a density. That is, when $(\mathcal{X}, \mathcal{A}, \mu)$ is a measure space with a measurable function $g: \mathcal{X} \to [0,\infty)$, and $\sigma_1 = g \# \mu$, where 
\begin{align} \label{eq: pushforward}
    g \# \mu (A) \coloneqq \mu(g^{-1}(A)) 
\end{align} for measurable $A \subseteq [0,\infty)$ and similarlly $\sigma_2 = f \# \nu$ with $f: \mathcal{Y} \to [0,\infty)$ measurable for $(\mathcal{Y}, \mathcal{B}, \nu)$ a measure space, and $\sigma = \sigma_2 - \sigma_1$. We say that $\bar{\mu} = g \# \mu$ is the pushforward of $\mu$ by $g$ or that $g$ transports $\mu$ to $\bar{\mu}$ when \eqref{eq: pushforward} holds for all measurable $A$.  Observe that $0 \prec f \# \nu - g \# \mu$ is equivalent to
\[
    \int \phi(g) d\mu \leq \int \phi(f) d\nu
\]
for all convex functions $\phi:[0,\infty) \to \mathbb{R}$ and in this case we may write for brevity $g \# \mu \prec f \# \nu $ in place of $0 \prec f \# \nu - g \# \mu$.
When $\mu=\nu$, we will further abbreviate to $f \prec_\mu g$.

We direct the reader to the textbooks \cite{marshall,shaked2007stochastic} for further background on the convex order and majorization.
We only stress here that our definition $0 \prec f \# \nu - g \# \mu$ is slightly more general and does imply $g \# \mu \prec f \# \nu $ in the commonly used sense.  

We adopt a formulation with signed measures for ease of use against integrability issues. The hypothesis that $\sigma$ is a finite signed measure majorizing $0$ implies that the positive measures $\sigma_+$ and $\sigma_-$ in the Hahn decomposition of $\sigma = \sigma_+ - \sigma_-$, possess the same finite measure, as $\int L d(\sigma_+ - \sigma_-) = 0$ for any linear function $L$, and in particular when $L = 1$.  However our definitions do not require that $f \# \nu$ and $g \# \mu$ themselves be finite measures. For example, with the signed measure formulation, one may consider and prove $g \prec_\mu f$ even when $\int g d\mu = \infty$, so long as $f-g$ is integrable, see Lemma \ref{lem:NP} below.

The notion of a distribution function, that we now introduce, will be useful in connecting our current investigations to previous literature, and for giving several equivalent formulations of Definition \ref{defn: Majorization} (Theorem \ref{thm: chong} below).
Let $(\mathcal{X},\mathcal{A},\mu)$ be a measure space ($\mu$ need not be a probability measure, we may often omit the $\sigma$-algebra). For a non-negative measurable function $g \colon \mathcal{X} \to [0,\infty)$, define its distribution function $G \colon [0,\infty) \to [0,\infty]$ by
$$
G(\lambda) \coloneqq \mu \left( \left\{ x \in \mathcal{X} : g(x) > \lambda  \right\} \right) .
$$

% The convex order is a notion of majorization or partial order on the density functions of      $(\mathcal{X},\mu)$ that reflects the ``spread'' of mass.

  We will demonstrate that the concept of majorization provides a simple and systematic means for understanding important integral inequalities. In fact, though not explicitly acknowledged in the literature, majorization techniques have been of significant recent interest for proving analytic and geometric inequalities.
In their seminal paper \cite{NP00}, Nazarov and Podkorytov introduced a very elementary but powerful lemma, that we may call Nazarov-Podkorytov's lemma in the sequel, based on distribution functions.  %This lemma, which we now introduce, is the starting point of our investigations.

\begin{lem}[\cite{NP00}] \label{lem:NP}
Let $f$ and $g$ be any two non-negative { measurable} functions on a measure space $(\mathcal{X},\mathcal{A},\mu)$. Let $F$ and $G$ be their distribution functions. Assume that both $F(\lambda)$ and $G(\lambda)$ are finite for every $\lambda >0$. Assume also that at some point $\lambda_o$ the difference $F-G$ changes sign from $-$ to $+$, \textit{i.e.},
$F(\lambda) \leq G(\lambda)$ for all $\lambda \in (0,\lambda_o)$ and $F(\lambda) \geq G(\lambda)$ for all $\lambda > \lambda_o$.
Let $S \coloneqq \{s > 0 : f^s - g^s \in \mathbb{L}^1(\mathcal{X},\mu)\}$.
Then 
% the function
% $$
% s \mapsto \varphi(s) \coloneqq \frac{1}{st_o^s} \int_\mathcal{X} (f^s - g^s) d\mu 
% $$
% is increasing on $S$. In particular,
if $\int_\mathcal{X} (f^{s_o}-g^{s_o})d\mu=0$,  $\int_\mathcal{X}(f^s-g^s)d\mu \geq 0$ for each $s > s_o$, $s \in S$. The equality may hold only if the functions $F$ and $G$ coincide.
\end{lem}

 It has attracted attention and found utility as a tool for delivering relatively simple arguments for $L^p$ norm comparisons between functions that would otherwise be very challenging to compare\footnote{In fact, Nazarov and Podkorytov proved a stronger result, that the function
$$
s \mapsto \varphi(s) \coloneqq \frac{1}{s\lambda_o^s} \int_\mathcal{X} (f^s - g^s) d\mu 
$$
is increasing on $S$.
However in applications, this monotonicity result has yet to find utility outside of the context of Lemma \ref{lem:NP}. Further, Nazarov-Podkorytov's lemma was used, to the best of our knowledge, only with $\mathcal{X}=(0,\infty)$, or $\mathbb{R}$.}.
% Namely in some practical situations, it is useful to know that, if $\int f d\mu \leq \int g d\mu$, then 
% $\int f^s d\mu \leq \int g^s d\mu$ for all $s \geq 1$.

The change of sign between the distribution functions $F$ and $G$ that appears in Lemma \ref{lem:NP}
is known in the literature as the \emph{single crossing property}. It is hard to give a sure attribution of this terminology. We could find its definition in a paper by Diamond and Estiglitz \cite[Page 3]{diamond}, in economy. However, such a property, with no specific name, was used earlier in probability theory, see \textit{e.g.}\ Karlin's book \cite{karlin}. Moreover, it appears that Nazarov-Podkorytov's lemma is essentially already contained, inter alia, in \cite{karlin-novikoff}, though this paper does not state it as clearly. 
{ Such a lemma is essentially part of the folklore  and is often re-derived on an ad-hoc basis. It is worth mentioning that Nazarov and Podkorytov themselves do no pretend at any novelty. Besides the papers already quoted above, let us mention \cite{MOP,MP} for a few other places where the reader can find similar statements.}

In fact, \cite{karlin-novikoff} also holds the majorization interpretation (Lemma \ref{lem-intro: single crossing implies majorization}) that was at the starting point of our investigations, even if we realized the existence of \cite{karlin-novikoff} only after the writing of the present article was complete\footnote{Confirming the law that mathematicians often rediscover results known for a long time... In the field of functional inequalities, the logarithmic Sobolev inequality of Shannon-Gross is a striking example.}.

We direct the reader to \cite{astashkin2021majorization} for recent extensions of Nazarov-Podkorytov's lemma to interpolation spaces using majorization.

The above lemma was the starting point of  Nazarov and Podkorytov's idea of unifying and
re-deriving, in a very elegant way, very deep results of Ball on sections of the unit cube \cite{ball} and of Haagerup on sharp constants in Kintchin's inequalities \cite{haagerup}.

In Ball's approach of the {cube} slicing problem, one key ingredient is to prove that
$$
\int_\mathbb{R} e^{-s\pi x^2/2} - \left| \frac{\sin(\pi x)}{ \pi x} \right|^s dx > 0 
$$
for all $s>2$, observe that for $s=2$ this is an identity. We may refer to it as \emph{Ball's integral inequality}. 
{Here and below $dx$ stands for the integration against the Lebesgue measure}. For the second statement the authors had to use a modified lemma (of the same spirit but with $F_*(y):=\mu(\{x : f(x) < y\})$ for $f$ bounded) to prove that
$$
\int_0^\infty \left(e^{-s\frac{x^2}{2}}-|\cos x|^s \right) \frac{dx}{x^{p+1}} \geq 0  
$$
for all $s \geq 2$, all $p \in [p_o,2)$ (for some well defined $p_o \simeq 1.85$, see \cite{haagerup}), an inequality due to Haagerup that we may call \emph{Haagerup's integral inequality}.

Besides Ball-Haagerup's integral inequalities, we mention that Nazarov and Podkorytov's lemma was exploited by different authors. Namely in a  refined version of Ball's integral inequality by K\"onig and Koldobsky \cite{konig-koldobsky}, with application to Busemann-Petty problem for the surface area, in his study of optimal Khinchin's inequality for $p \in (2,3)$
by Mordhorst \cite{Mordorst}.
It
was also used in \cite{MMR20} to compare the Fourier transform of a Bernoulli random variable to a Gaussian.
%Namely, for $X$ a Bernoulli with variance $\sigma^2$, then $q \geq 1$ implies
%    \begin{align*}
 %       \frac{1}{2 \pi} \int_{-\pi}^\pi |\mathbb{E}e^{itX}|^q dt \leq \frac1 { \sqrt{6 \sigma^2 q}} \int_0^{ \sqrt{6 \sigma^2 q} } e^{-t^2/2} dt.
 %   \end{align*}

The next lemma, that goes back to \cite{karlin-novikoff}, gives a majorization interpretation of Nazarov and Podkorytov's lemma. Our aim in this article is to introduce, in the Lebesgue case, an alternative based on transport arguments. Both constitute practical tools to prove majorization (between two integrable functions) that is, in many situations, a hard task.

\begin{lem}[\cite{karlin-novikoff}] \label{lem-intro: single crossing implies majorization}
Let $f$ and $g$ be any two non-negative functions on a measure space  $(\mathcal{X},\mathcal{A},\mu)$ satisfying {$\int (f - g)d\mu = 0$}. Let $F$ and $G$ be their distribution functions. 
%Assume that both $F(\lambda)$ and $G(\lambda)$ are finite for every $\lambda >0$. {Cyril: this assumption is guaranteed by the fact that $f$ and $g$ are integrable and non-negative.}
Assume that at some point $\lambda_o$ the difference $F-G$ changes sign from $-$ to $+$, \textit{i.e.},
$F(\lambda) \leq G(\lambda)$ for all $\lambda \in (0,\lambda_o)$ and $F(\lambda) \geq G(\lambda)$ for all $\lambda > \lambda_o$. {Then 
$$
g  \prec_\mu f , 
$$
and consequently $\int_{\mathcal{X}} (\phi(f) - \phi(g)) d \mu \geq 0$ when $\phi$ is convex}.
\end{lem}

In other words, {\emph{the single crossing property implies majorization}.
In the next lemma we will prove that \emph{transport implies majorization} 
leading to the same conclusion under very different conditions (see Remark \ref{rem:tm}).}

We anticipate on the fact that the conclusion of Nazarov and Podkorytov's lemma (Lemma \ref{lem:NP}, for integrable functions) is then a straight forward consequence of Lemma \ref{lem-intro: single crossing implies majorization} thanks to a well-known equivalent formulation of the majorization (see Section \ref{sec:majorization}).
%It should be noticed that, since the definition of the majorization involves only integrable functions, the above lemma must deal also with integrable functions only.

In order to state our second lemma, we { need to introduce some notations and definition. We} use the notation $|T'(x)|$ to denote the absolute value of the determinant of the linear map $T'(x)$ obtained as the derivative of a function $T: \mathbb{R}^n \to \mathbb{R}^n$. 
{ 

\begin{defn}[Change of variables]
Let $\phi,\psi \colon \mathbb{R}^n \to \mathbb{R}_+$ be measurable with $\int \phi(x) dx < \infty$. We say that a function $T: \mathbb{R}^n \to \mathbb{R}^n$ is a change of variables between $\phi$ and $\psi$ if $T$ is locally-Lipschitz, injective and satisfies Lebesgue almost everywhere $\phi = \psi(T)|T'|$.
\end{defn}
The definition implies that, for any measurable function $h$, it holds
$\int h(T(x)) \phi(x) dx = \int h(x) \psi(x) dx$. 
Change of variables are related to transport of mass.
We refer the reader to Section \ref{sec:majorization} for comments on the existence and regularity of changes of variables  and to the book by Villani \cite{villani2009optimal} for an introduction to the field of optimal transport. We stress that the existence of a change of variables implies the conservation of mass $\int \phi(x) dx= \int \psi(x)dx$.

In the sequel $\infess$ and $\supess$ denotes the essential infimum and essential supremum, respectively, namely for a measurable function $f \colon \mathbb{R}^n \to \mathbb{R}$, 
$\supess f = \inf \{a : f(x) \leq a \mbox{ for Lebesgue almost all } x \in \mathbb{R}^n\}$ (with $\supess f = \infty$ if the set is empty) and similarly for $\infess$.}

\begin{lem} \label{lem-intro: King Transport-Majorization}
Let $\mu(dx) = u(x) dx$ be a (not necessarily finite)  measure on $\mathbb{R}^n$, $f,g : \mathbb{R}^n \to \mathbb{R}_+$ be non-negative {measurable} such that $\int g d\mu < \infty$, and {assume that there exists  a change of variables $T: \mathbb{R}^n \to \mathbb{R}^n$ from $ug$ and $uf$} satisfying ${\infess}  \frac{u}{u(T)|T'|}  \geq 1$. {Then,
$$
g \prec_\mu f.
$$
}
% . Then 
% for all convex functions $\phi: \mathbb{R}_+ \to \mathbb{R}$
% \begin{align*}
%     \int \phi(g) d\mu \leq \left({\infess} \frac{u}{u(T)|T'|}\right) \int \phi\left(f / {\infess} \frac{u}{u(T)|T'|}  \right) d \mu
% \end{align*}
% provided the infimum belongs to $(0,\infty)$.  Further if ${\infess}  \frac{u}{u(T)|T'|}  \geq 1$, then 
\end{lem}

{A more general version of the above lemma will be stated in Section \ref{sec:majorization-different-measure}.}

As a simple but useful application, let us restate the conclusion of the above result in the spirit of Nazarov and Podkorytov's lemma (take 
$u \equiv 1$ and $\varphi(x)=x^{s/s_o}$ in Lemma \ref{lem-intro: King Transport-Majorization}, details are left to the reader).

 \begin{lem} \label{lem-intro: transport}
 	Let $f,g: \mathbb{R}^n \to [0,\infty)$ be { measruable} such that there exists $T: \mathbb{R}^n \to \mathbb{R}^n$ and $s_o>0$ such that { $T$ is a change of variables from $g^{s_o}$ and $f^{s_o}$, with} $|T'(x)| \leq 1$ and $\int g^{s_o}(x) dx = \int f^{s_o}(x)dx < \infty$. Then
 	$\int g^s(x) dx \leq \int f^s(x) dx$ for all $s \geq s_0$.
 \end{lem}

% As a straight forward application, consider $f,g: \mathbb{R}^n \to [0,\infty)$ such that there exists $T: \mathbb{R}^n \to \mathbb{R}^n$ such that $g(x) = f(T(x)) |T'(x)|$  with $|T'(x)| \geq 1$. Then if $\int f^{s_0}(x) dx = \int g^{s_0}(x) dx < \infty$	for $s_0 > 0$, the above lemma applied to 
% $\mu(dx)=f^{s_o}(x)dx$ and $\nu(dx)=g^{s_o}(x)dx$ (\textit{i.e.}\ $u \equiv v = 1$) and $\varphi(x)=x^{s/s_o}$ implies  $\int f^s(x) dx \geq \int g^s(x) dx$ for all $s \geq s_0$.

As a second application of Lemma \ref{lem-intro: King Transport-Majorization}, with the {help} of the celebrated Cafarelli's contaction Theorem, we will prove that any strongly log-concave density majorizes the standard Gaussian density (see Section \ref{sec:majorization}), a result that is very natural and, to the best of our knowledge, was not known.
Related to information theory, as a corollary, we will deduce that the Gaussian distribution maximizes, among all strongly log-concave distributions, the R\'enyi and Tsallis entropies of all orders.  The R\'enyi entropy \cite{renyi1961measures} unifies the Shannon, min, Hartley, and collision entropies, and has been long used in information theory, see the survey \cite{van2014renyi} for more background.  The Tsallis entropy \cite{tsallis1988possible} is an alternative generalization of the Shannon entropy proven useful in statistical mechanics in the last few decades.

To further demonstrate the efficacy of  Lemma \ref{lem-intro: King Transport-Majorization} (and in practice Lemma \ref{lem-intro: transport}), in Section \ref{sec:transportBall} we give simple proofs of some integral inequalities previously derived through the method of Nazarov and Podkorytov.  We also use Lemma \ref{lem-intro: transport} to derive the following main result.

\begin{thm}\label{thm: discrete Ball's inequality}
     For $p \geq 2$, and $2 \leq n \in \mathbb{N}$,
    \begin{align*}
        \int_{-\frac 1 2}^{\frac 1 2} \left|\frac{\sin( n\pi x)}{n \sin \pi x} \right|^p dx %\leq \int_{-\infty}^\infty \left(e^{- \pi (\ell^2 - 1) x^2/2}\right)^p dx =
        < \sqrt{\frac{2}{p \ (n^2-1)}}
    \end{align*}
\end{thm}

We note that using the method of Nazarov-Podkorytov, Theorem \ref{thm: discrete Ball's inequality} was obtained for $n \geq N$ for a fixed $N \in \mathbb{N}$ in an unpublished work \cite{MMX2017unpublished}.
As a corollary of Theorem \ref{thm: discrete Ball's inequality}, we obtain sharp upper-bounds on the cardinality of $A_k \coloneqq \{ z \in \mathbb{Z}^n : z_i \in [0,n_i], z_1 + \cdots + z_n = k \}$.  The continuous version of this problem, upper-bounds on the volume of $\tilde{A}_\lambda \coloneqq \{ x \in \mathbb{R}^n : x_i \in [0,n_i], x_1 + \cdots + x_n = \lambda \}$ is equivalent to upper-bounds on slices of the cube, and it is in this sense that we consider Theorem \ref{thm: discrete Ball's inequality} to be a discrete analog of Ball's integral inequality.  Moreover, it will be shown by letting $n \to \infty$ in Theorem \ref{thm: discrete Ball's inequality} one {recovers} Ball's integral inequality.

{
As a final remark, we mention that some of our results can be extended to functions $f$ and $g$ leaving on different measure spaces $(\mathcal{X}, \mathcal{A}, \mu)$ and $(\mathcal{Y}, \mathcal{B}, \nu)$. For this we will need to restrict to convex functions vanishing at $0$ and to introduce the appropriate corresponding definition of majorization (see Section  \ref{sec:majorization-different-measure}).}

% As a final remark, observe that since the notion of majorization is defined only for integrable functions we need to assume integrability in all our statements{. In fact renormalizing if necessary, we are dealing all along the paper with probability density} while for instance in Lemma \ref{lem:NP} it is only asked that $\int (f^{s_o}-g^{s_o})d\mu < \infty$. As a consequence, Haagerup's integral inequality is excluded from our analysis.
% We believe that a proper notion of majorization could be defined for non-integrable functions, leading to more general results of  Nazarov-Podkorytov's type. However we decided to restrict ourselves to integrable functions for simplicity of the exposition and because our results already lead to a series of non-trivial applications. 
% We also believe that both our new lemma and its applications are of interest for a wide community that includes Information Theory, Probability Theory and Banach space geometry.

{
\subsection*{Acknowledgement}
We warmly thank an anonymous referee and the editor for their comments and suggestions that strongly improve the presentation of this paper.
We also thank Dario Cordero-Erausquin, Matthieu Fradelizi, Mokshay Madiman and Paul-Marie Samson for useful discussions on the literature and on the topic of this paper.
}

\tableofcontents

%%%%%%%%%%%%%%%%%%%%%%%%%%%%%%%%%%%%%%%%%%%%%%%%%%%%%%%%%%%%%%%%%%%%%%%%%%%%%%%%%
%%%%%%%%%%%%%%%%%%%%%%%%%%%%%%%%%%%%%%%%%%%%%%%%%%%%%%%%%%%%%%%%%%%%%%%%%%%%%%%%%
%%%%%%%%%%%%%%%%%%%%%%%%%%%%%%%%%%%%%%%%%%%%%%%%%%%%%%%%%%%%%%%%%%%%%%%%%%%%%%%%%

\section{Majorization and Transport}\label{sec:majorization}

This section collects some aspects of majorization related to  Nazarov and Podkorytov’s lemma.
We first recall some basic properties of majorization. Then we prove Lemma 
\ref{lem-intro: single crossing implies majorization}. In the next sub-sections we deal with the transport approach of majorization, {and, together with Cafarelli's contraction Theorem,  with strongly log-concave distributions. Finally, in the last sub-section we extend some of our results to functions $f$ and $g$ leaving on different measure spaces.}

{All along the section $(\mathcal{X}, \mathcal{A},\mu)$ denotes a measure space and $g$ and $f$ non-negative measurable functions $g,f:\mathcal{X} \to \mathbb{R}_+$.  Their respective distribution functions $G$ and $F$ denote $G(\lambda) = \mu \{ g > \lambda \}$ and $F(\lambda) = \mu \{ f > \lambda\}$.
}

\subsection{Basics on majorization}
%Recall the definition of Majorization.

% \begin{defn}[Majorization]
% For $f$ and $g$ non-negative integrable functions, we write $f \prec g$ and say $g$ \emph{majorizes} $f$ when $t > 0$ implies,
% \begin{align}
%     \int_t^\infty F(\lambda) d \lambda \leq \int_t^\infty G(\lambda) d \lambda.
% \end{align}
% %When $X$ and $Y$ are $\mathcal{X}$ valued random variables with respective densities $f$ and $g$ with respect to $\mu$, we write $X \prec Y$ when $f \prec g$.
% \end{defn}

For more background on majorization see \cite{marshall}.
The following theorem { is a reformulation in terms of a single signed measure, of well known, equivalent, useful descriptions of majorization, see for instance \cite{chong} for a classical reference.

\begin{thm} \label{thm: chong}
For a signed measure $\sigma$ on $[0,\infty)$ such that $\int_0^\infty L(x) d \sigma = 0$ for any affine function\footnote{We assume tacitly that the integrals are well defined, which will only be true if $\int L d|\sigma| < \infty$} $L(x)=ax+b$, the following are equivalent;

\begin{enumerate}
    \item  \label{eq: chong signed} $$\sigma \succ 0$$
    \item \label{eq: chong distribution} For $t \geq 0$, $$\int_t^\infty \sigma(\lambda,\infty) d \lambda \geq 0.$$
    \item \label{eq: chong hockey stick}
    For $t \geq 0$, $$\int_0^\infty [ x - t]^+ d\sigma(x) \geq 0 .$$
\end{enumerate}
\end{thm}

Note that $\int_0^\infty L(x) d\sigma(x) = 0$ is a necessary condition for $\sigma \succ 0$ since $L$ and $-L$ are both convex functions.
\begin{proof}
    Note that the finite assumptions on $\sigma$ justify the change of order of integration,
    \[
        \int_t^\infty \sigma[\lambda,\infty) d\lambda = \int_0^\infty \left( \int_0^\infty \mathbbm{1}_{\{x \geq \lambda \geq t \}} d \lambda \right) d\sigma(x) = \int_0^\infty [x-t]_+ d\sigma(x),
    \]
    so that \eqref{eq: chong distribution} $\iff$ \eqref{eq: chong hockey stick}.  To prove \eqref{eq: chong hockey stick} $\Rightarrow$ \eqref{eq: chong signed}  the Taylor series expansion of a smooth convex function $\phi$ as
    \[
        \phi(x) = \phi(0) + x \phi'(0) + \int_0^\infty [x-t]^+ \phi''(t) dt,
    \]
    we have
    \[
        \int_0^\infty \phi(x) d\sigma(x) = \int_0^\infty \left( \phi(0) + x \phi'(0) \right) d\sigma(x) + \int_0^\infty \left(  \int_0^\infty [x-t]^+ d\sigma(x) \right) \phi''(t) dt \geq 0.
    \]
By approximation the result follows for general convex $\varphi$.  This completes the proof since \eqref{eq: chong signed} $\Rightarrow$ \eqref{eq: chong hockey stick} is immediate.
\end{proof}

}

\subsection{Nazarov and Podkorytov’s lemma as a consequence of majorization} 

We start by  proving Lemma \ref{lem-intro: single crossing implies majorization}.

% \begin{lem} \label{lem: single crossing implies majorization}
% For $f,g \geq 0$ with respective distribution functions $F$ and $G$ such that there exists $\lambda_0$ such that $G(\lambda) \geq F(\lambda)$ holds for $\lambda \geq \lambda_0$ and $G(\lambda) \leq F(\lambda)$ for $\lambda \leq \lambda_0$, then $f \prec g$.
% \end{lem}

\begin{proof}[Proof of Lemma \ref{lem-intro: single crossing implies majorization}] {The signed measure $\sigma = f \# \mu - g \# \mu$, satisfies
\[
    \int_0^\infty \psi d\sigma = \int_{\mathcal{X}} \psi(f) - \psi(g) d\mu
\]
for measurable $\psi$ for which either side of the equality is well defined. Hence, it follows that $\int_0^\infty 1 d\sigma = \int_0^\infty x d\sigma = 0$, and hence $\int_0^\infty L d\sigma  = 0$ for any $L(x) = ax + b$.  Moreover, $\sigma(t, \infty) = F(t) - G(t)$.}
Thus it follows from the assumptions on $F-G$ that $\Phi(t) \coloneqq \int_t^\infty \sigma(\lambda,\infty) d \lambda$ is non-decreasing for $t \leq \lambda_0$ and non-increasing for $t \geq \lambda_0$, and since $\lim_{t \to \infty} \Phi(t) = \Phi(0) = \int (g - f) d\mu  = 0$ we have $\Phi(t) \geq 0$. {Hence, $0 \prec  \sigma =g \# \mu - f \# \mu $ {by item \eqref{eq: chong distribution} of Theorem \ref{thm: chong}}}.
\end{proof}

As an immediate consequence of the majorization property, we re-prove 
Nazarov and Podkorytov’s lemma in the following form.

\begin{cor}[Nazarov-Podkorytov \cite{NP00}, Karlin-Novikoff \cite{karlin-novikoff}] \label{cor:NP}
Let $g$ and $f$ be two non-negative { measruable} functions on  $(\mathcal{X},\mathcal{A},\mu)$. Let $G$ and $F$ be their distribution functions. 
Assume that {$\int f^{s_o} - g^{s_o} d\mu =0$} for some $s_o >0$. 
Assume also that at some point $\lambda_o$,
$F(\lambda) \leq G(\lambda)$ for all $\lambda \in (0,\lambda_o)$ and $F(\lambda) \geq G(\lambda)$ for all $\lambda > \lambda_o$.    
Then {
    \begin{align*}
        \int (f^s - g^s) d \mu \geq  0
    \end{align*}
    }
    for all $s \geq s_0$.
\end{cor}

\begin{proof}
Writing, $G_{s_0}(t) = \mu \{ g^{s_0} > t \}$, and $F_{s_0}(t) = \nu \{ f^{s_0} > t\}$ then $F_{s_0}(t) = F(t^{\frac 1 {s_0}})$ and $G_{s_0}(t) = G(t^{\frac 1 {s_0}})$, so that $F_{s_0}$ and $G_{s_0}$ have a single crossing at $\lambda_0^{\frac 1 {s_0}}$ and satisfy the hypothesis of 
Lemma \ref{lem-intro: single crossing implies majorization}, and we have {$g^{s_0} \prec_\mu f^{s_0}$}.  Consequently, by Theorem \ref{thm: chong} for $\phi$ convex  we have
{
\begin{align} \label{eq: Joe applied to NP}
    \int (\phi(f^{s_0}) - \phi(g^{s_0})) d\mu \geq 
0 .
\end{align}
}
For $s \geq s_0$ taking $\phi(x) = x^{\frac s {s_0}}$ gives the result.
\end{proof}

\subsection{Majorization via transport}
This section is dedicated to the proof of Lemma \ref{lem-intro: King Transport-Majorization}.

{
Before moving to the proof, let us make some comments on the existence and regularity of the map $T$.

The change of variable assumption guarantees the existence of a map  $T$ that transports $gu$ to $fv$: $T \# gu = fv$. 
Set $\tilde \mu=g\mu/\int g d\mu$ and $\tilde \nu = f \nu/\int f d\nu$ 
to turn $g\mu$ and $f\nu$ into  probability measures on $\mathbb{R}^n$.
We stress that $T$ can be defined arbitrarily on any set of Lebesgue measure $0$. Indeed, if $T=T_o$ for almost all $x \in \mathbb{R}^n$,
$\int \int h(T) d\tilde \mu = \int_{\{T=T_o\}} \int h(T) d\tilde \mu
= \int_{\{T=T_o\}} \int h(T_o) d\tilde \mu = \int \int h(T_o) d\tilde \mu$.

Brenier's theorem \cite{brenier} asserts that, as soon as $\mathcal{T}_2(\tilde \mu,\tilde \nu) < \infty$ (that is a very mild assumption), where 
$\mathcal{T}_2$ is the optimal transport associated to the quadratic cost (see \cite{villani2009optimal}), the  transport map $T$ exists and is the gradient of a convex function $\Phi \colon \mathbb{R}^n \to \mathbb{R}\cup \{\infty\}$ with $\tilde \mu (\{ \Phi < \infty\})=1$. Furthermore, $T'$ exists and satisfies the Monge-Amp\`ere equation in the Aleksandrov sense.  We refer the reader to \cite{villani03} for more details.
}

The invertibility of the map $T$ is not guaranteed in general in the transport theory, even for the Brenier map. However, since $f\nu$ and $g\mu$ are absolutely continuous with respect to the Lebesgue measure,  $T = \nabla \Phi$, with $\Phi$ convex, has an inverse  and in fact $T^{-1}= \nabla \Phi^*$ where $\Phi^*$ is the Legendre transform of $\Phi$ and {$T^{-1}\# (f\nu)=g\mu$}. {In particular, in our setting,  the Brenier map $T=\nabla \Phi$ is always a change of variables.}

In dimension $1$ the situation is simpler since $T = F^{-1}\circ G$, with $F$ and $G$ the distribution functions of $f\nu$ and $g\mu$, is increasing, $T^{-1}=G^{-1}\circ F$ is always well-defined.

%Since optimal transport theory is well understood in the more general setting of Riemaniann manifolds  as well, the lemma can probably be extended to that setting too.

{
\begin{proof}[Proof of Lemma \ref{lem-intro: King Transport-Majorization}]
Let $t \geq 0$. Since $\phi \colon x \mapsto [x-t]^+$ is a convex function vanishing at $0$, $\phi(sx) \leq s \phi(x)$ for any $s \in (0,1)$ and $x \geq 0$. Therefore 
\begin{align*}
\int [g(x)-t]^+d\mu 
& =
\int_{\{u \neq 0\}} \left[\frac{u(T(x))f(T(x))|T'(x)|}{u(x)}-t \right]^+u(x)dx \\
& \leq 
\int [f(T(x))-t]^+ u(T(x))|T'(x)| dx \\
& =
\int [f(x)-t]^+d\mu.
\end{align*}
Since $f$ and $g$ are integrable, the above inequality is equivalent to
$$
\int_0^\infty [x-t]^+ d\sigma \geq 0
$$
for $\sigma = f \# \mu - g \# \mu$. The expected result follows from Theorem \ref{thm: chong} Item \eqref{eq: chong hockey stick}.
\end{proof}
}

% that we state below in a more general form, together with some useful consequences.  {As a warm up let us first state and prove a simpler version of Lemma \ref{lem-intro: King Transport-Majorization} that serves to provide some intuition for the more complicated version below.}

\subsection{Strongly log-concave distributions and majorization}

We prove here that any strongly log-concave density majorizes the Gaussian density and that the Gaussian density maximizes the R\'enyi and Tsallis entropies among all strongly log-concave densities. To that aim, we need first to recall the definition of strongly log-concave densities and Cafarelli's contraction Theorem.

We denote by $\gamma_n$ the standard Gaussian measure on $\mathbb{R}^n$, with density ${g}(x) \coloneqq \frac{d\gamma_n}{dx} =  \frac{1} {(2 \pi)^{\frac n 2}} e^{-|x|^2/2} $.

\begin{defn}
A probability density function ${f} : \mathbb{R}^n \to [0,\infty)$ is strongly log-concave when $\nu(dx) = {f}(x) dx$  is log-concave with respect to $\gamma_n$.
In other words, there exists a convex function $V$ such that $\nu(dx) = e^{-V(x)} \gamma_n(dx)$.
\end{defn}

\begin{thm}[Caffarelli \cite{caffarelli2000monotonicity, caffarelli2002monotonicity}] \label{thm: Caffarelli} Let $\nu${,  with density $f$ with respect to the Lebesgue measure,} be strongly log-concave{. T}hen there exists a  $1$-Lipschitz {change of variables} $T$ {from $g$ to $f$} such that $T = \nabla \varphi$, for $\varphi$ convex.
\end{thm}

The existence of a connection between majorization and Caffarellli's contraction theorem seems to go back to Harg\'e \cite{harge} (see  \cite{gozlan-juillet,fathi2020aproof} for more recent results in this direction). Here we may put together our transport approach of the majorization in Lemma \ref{lem-intro: King Transport-Majorization} and the latter theorem to get the following natural statement.

\begin{cor} \label{cor: strongly log-concave majorizes Gaussian}
    If {$f$} is a strongly log-concave density function,  then 
    {$g \prec_{dx} f$} where {$g$} is the standard Gaussian density defined above.
\end{cor}

\begin{proof}
    By Theorem \ref{thm: Caffarelli}, there exists a $1$-Lipschitz {change of variables} $T: \mathbb{R}^n \to \mathbb{R}^n$ such that $T = \nabla \varphi$, for a convex $\phi \colon \mathbb{R}^n \to \mathbb{R}$ such that $T \# \gamma_n = \nu$ (with $\nu(dx)={f}(x)dx$).  As the Hessian of a convex function $T'(x)$ is symmetric and positive definite, thus it has non-negative eigenvalues.  Since $T$ is Lipschitz, its eigenvalues are all bounded by $1$, and hence $|T'(x)| \leq 1$ for all $x$.  
    {Applying Lemma \ref{lem-intro: King Transport-Majorization} 
    leads to the desired conclusion.}
\end{proof}

{
\begin{remark}
For comparison let us mention that Harg\'e proved in \cite{harge} that $0 \prec \gamma_n - \nu$ (under the assumption that $\int x d\nu =0$), which amounts to saying that, for all convex  
function $\phi$ it holds
$\int \phi d\gamma_n \geq \int \phi d\nu$,
while our conclusion $g \prec_{dx} f$ reads $\int \phi(g) dx \leq \int \phi(f) dx$ for all $\phi$ convex.
\end{remark}
}

As a corollary, we will prove that the Gaussian distribution maximizes the R\'enyi and Tsallis entropies of all orders (among the set of all  strongly log-concave densities). Let us recall some definition.

\begin{defn}[Tsallis \cite{tsallis1988possible}/R\'enyi \cite{renyi1961measures} Entropy]
For a probability density function ${f}:\mathbb{R}^n \to \mathbb{R}_+$, and $q \in (0,1) \cup (1,\infty)$ we denote the $q$-Tsallis entropy by
\begin{align*}
    S_q({f}) \coloneqq \frac{ \int {f}^q(x) dx -1}{1-q}.
\end{align*}
For $q = 1$, $S_1({f}) \coloneqq h({f}) = - \int {f}(x) \log {f}(x) dx$ is the Shannon entropy.  Via continuous extension, one can define $S_0({f}) = | \{ {f} > 0 \} | - 1$ and $S_\infty({f}) = 0$ when $\| {f} \|_\infty \leq 1$ and $S_\infty({f}) = - \infty $ otherwise.

We denote the $q$-R\'enyi entropy, 
\begin{align*}
    h_q({f}) \coloneqq \frac{ \log \int {f}^q(x) dx}{1-q},
\end{align*}
and  define $h_1({f}) \coloneqq h({f})$, $h_0({f}) = |\{ {f} > 0\}|$, $h_\infty({f}) = \|g\|_\infty$.
\end{defn}

Observe that % $\exp[(1-q) h_q(f)] = (1-q) S_q(f) - 1$, 
$S_q({f}) 
%= \frac{ \exp[(1-q) h_q(f)] + 1}{1-q}
= \Psi_q( h_q({f}))$ where 
\begin{align} \label{eq: Psi function}
\Psi_q(x)  \coloneqq \frac{ \exp{[(1-q) x}] {-} 1}{1-q}, \qquad x \in \mathbb{R}
\end{align}
is a strictly increasing function.

We are in position to state our corollary.
\begin{cor}
    Let $ f \colon \mathbb{R}^n \to \mathbb{R}$ be a strongly log-concave probability density, then for $q \in [0,\infty]$
    \begin{align*}
        h_q({f}) \leq h_q({g})
    \end{align*}
    and
    \begin{align*} 
        S_q({f}) \leq {S}_q({g})
    \end{align*}
    where ${g}$ is the Gaussian density defined above. 
\end{cor}

\begin{proof}
In light of the one-to-one relationship between $h_q$ and $S_q$ given in \eqref{eq: Psi function}, it suffices to prove the result for the R\'enyi entropy.
When $q=0$ there is nothing to prove since $h_q({g}) = \infty$.  For $q \in (0,1)$, the function $\varphi(x) = - x^q$ is convex and hence by Corollary \ref{cor: strongly log-concave majorizes Gaussian}, {$g \prec_{dx} f$}, and applying Theorem \ref{thm: chong},
\begin{align*}
    \int \varphi({g(x)})dx \leq \int \varphi({f(x)})dx,
\end{align*}
which gives $\int f^{q} {\leq} \int g^{q}$, or $h_q(f) {\leq} h_q(g)$.  When $q = 1$, the convex function $\varphi(x) = x \log x$, gives the result for the Shannon entropy.  Again, when $q > 1$, take $\varphi(x) = x^q$.  For $q = \infty$, observe that ${f}$ strongly log-concave is more than sufficient to give $\lim_{q \to \infty} h_q({f}) = h_\infty({f})$ from which  the result follows, ending the proof of the Corollary. 
\end{proof}

As a final remark note that if $X$ is a random variable and $Y = T(X)$ for $|T'(x)| \geq 1$, then if $Y$ has density $f$, $X$ has density $g(x) = f(T(x))|T'(x)|$ and
$h_q(T(X)) \geq h_q(X)$, and similarly for $S_q$.

\subsection{{Extensions to different measure spaces}} \label{sec:majorization-different-measure}

{
This section collects some generalization of the previous results. In particular, we will extend Lemma \ref{lem-intro: King Transport-Majorization} to functions $f$ and $g$ leaving on different spaces and relax the condition $\infess \frac{u}{v(T)|T'|} \geq 1$.

Here $(\mathcal{X}, \mathcal{A}, \mu)$ and $(\mathcal{Y}, \mathcal{B}, \nu)$ denote measure spaces and $g$ and $f$ non-negative measurable functions $g:\mathcal{X} \to \mathbb{R}_+$ and $f: \mathcal{Y} \to \mathbb{R}_+$.  Their respective distribution functions $G$ and $F$ denote $G(\lambda) = \mu \{ g > \lambda \}$ and $F(\lambda) = \nu \{ f > \lambda\}$. 

The measure spaces will not be assumed to have the same measure ($\mu(\mathcal{X})$ and $\nu(\mathcal{Y})$ need not be equal).
This will not guarantee anymore that $\int_0^\infty L(x) d \sigma = 0$ for any affine function. In particular,  the assumption $\int_0^\infty d\sigma = 0$ will not be satisfied leading us to consider only convex functions vanishing at $0$ and therefore a modified definition of majorization. Similarly to
Theorem \ref{thm: chong}, the following holds (the proof being similar, we left it to the reader).

\begin{thm} \label{th: Chong vanish}
For a signed measure $\sigma$ on $[0,\infty)$ such that that $\int_0^\infty x d \sigma = 0$, the following are equivalent;
\begin{itemize}
    \item[(i)]  \label{eq: chong signed without matching mass} For $\phi$ convex with $\phi(0) = 0$,
    $$
        \int_0^\infty \phi(x) d\sigma(x) \geq 0.
    $$
    \item[(ii)] \label{eq: chong distribution without matching mass} For $t \geq 0$, $$\int_t^\infty \sigma(\lambda,\infty) d \lambda \geq 0.$$
    \item[(iii)] \label{eq: chong hockey stick without matching mass}
     For $t \geq 0$, $$\int_0^\infty [ x - t]^+ d\sigma(x) \geq 0 .$$
\end{itemize}
\end{thm}
We will write $0 \prec_0 \sigma$ when any of the above are satisfied (and accordingly $g \# \nu \prec_0 f \# \mu$).

In \cite{NP00} Nazarov and Podkorytov’s lemma (Lemma \ref{lem:NP}) is stated in the case that $\mu = \nu$. However their proof can be easily adapted in the case that $\nu$ differs from $\mu$ and $f$ and $g$ are integrable. In such a setting, the following counterpart of Lemma \ref{lem-intro: single crossing implies majorization} and Corollary \ref{cor:NP} hold.
Their proof are left to the reader.

\begin{lem}[\cite{karlin-novikoff}] \label{lem: single crossing implies majorization vanish}
Let $g$ and $f$ be two non-negative { measruable} functions on  $(\mathcal{X},\mathcal{A},\mu)$ and $(\mathcal{Y}, \mathcal{B}, \nu)$ respectively. 
Let $G$ and $F$ be their distribution functions. 
Assume that $\int g d\mu = \int f d\mu < \infty$.
Assume also that at some point $\lambda_o$ the difference $F-G$ changes sign from $-$ to $+$, \textit{i.e.},
$F(\lambda) \leq G(\lambda)$ for all $\lambda \in (0,\lambda_o)$ and $F(\lambda) \geq G(\lambda)$ for all $\lambda > \lambda_o$. Then 
$$
g \# \mu  \prec_0  f \# \nu , 
$$
and consequently $\int_{\mathcal{X}} \phi(g) d\mu \leq \phi(f)) d \nu $ when $\phi$ is convex and $\phi(0)=0$.
\end{lem}

\begin{cor}[Nazarov-Podkorytov \cite{NP00}, Karlin-Novikoff \cite{karlin-novikoff}] \label{cor:NPvanish}
Let $g$ and $f$ be two non-negative { measruable} functions on  $(\mathcal{X},\mathcal{A},\mu)$ and $(\mathcal{Y}, \mathcal{B}, \nu)$ respectively. Let $G$ and $F$ be their distribution functions. 
Assume that $\int f^{s_o} d\nu = \int g^{s_o} d\mu< \infty$ for some $s_o >0$. 
Assume also that at some point $\lambda_o$,
$F(\lambda) \leq G(\lambda)$ for all $\lambda \in (0,\lambda_o)$ and $F(\lambda) \geq G(\lambda)$ for all $\lambda > \lambda_o$.    
Then
    \begin{align*}
        \int g^s d \mu \leq \int f^s d \nu
    \end{align*}
    for all $s \geq s_0$.
\end{cor}

% \begin{proof}
% Writing, $G_{s_0}(t) = \mu \{ g^{s_0} > t \}$, and $F_{s_0}(t) = \nu \{ f^{s_0} > t\}$ then $F_{s_0}(t) = F(t^{\frac 1 {s_0}})$ and $G_{s_0}(t) = G(t^{\frac 1 {s_0}})$, so that $F_{s_0}$ and $G_{s_0}$ have a single crossing at $\lambda_0^{\frac 1 {s_0}}$ and satisfy the hypothesis of 
% Lemma \ref{lem-intro: single crossing implies majorization}, and we have $g^{s_0} \# \mu \prec f^{s_0} \# \nu$.  Consequently, by Theorem \ref{thm: chong} for $\phi$ convex and non-decreasing, satisfying $\phi(0) = 0$ we have
% \begin{align} \label{eq: Joe applied to NP}
%     \int \phi(g^{s_0}) d\mu \leq \int \phi({f^{s_0}}) d\nu.
% \end{align}
% For $s \geq s_0$ taking $\phi(x) = x^{\frac s {s_0}}$, gives the result.
% \end{proof}

Next we turn to a generalization of Lemma \ref{lem-intro: King Transport-Majorization}.

\begin{lem} \label{lem: King Transport-Majorization vanish}
Let $\mu$ and $\nu$ be measures on $\mathbb{R}^n$ (not necessarily finite), such that $\mu(dx) = u(x) dx$ and $\nu(dx) = v(x) dx$, $f,g : \mathbb{R}^n \to \mathbb{R}_+$ be non-negative and { measruable} such that $\int g d\mu < \infty$, and { assume that there exists a change of variables $T: \mathbb{R}^n \to \mathbb{R}^n$ from $gu$ to $fv$.}
Then for all convex functions $\phi: \mathbb{R}_+ \to \mathbb{R}$, with $\phi(0) = 0$,
\begin{align*}
    \int \phi(g) d\mu \geq  A  \int \phi(f/A) d \nu
\end{align*}
if
\begin{align*}
    A \coloneqq \supess \frac{u}{v(T)|T'|} \in (0,\infty).
\end{align*}
Further if $A \leq 1$, $f \# \nu \prec_0 g \# \mu$.\\
Also,
\begin{align*}
    \int \phi(g) d\mu \leq  A'  \int \phi(f/A') d \nu
\end{align*}
if 
\begin{align*}
    A' \coloneqq \infess \frac{u}{v(T)|T'|} \in (0,\infty),
\end{align*}
in which case $g \# \mu \prec_0 f \# \nu$ if $A' \geq 1$.
\end{lem}

% \begin{remark} \label{rem:simple}
% We observe that, when $\mu$ and $\nu$ are the Lebesgue measure (\textit{i.e.}\ $u =v = 1$) and when one assumes the existence of a change of variable $T$ from $g$ to $f$ satisfying $|T'| \leq 1$,  the conclusion of the second part of the lemma (\textit{i.e.} $g \# dx \prec f \# dx$) can be obtained very easily by a direct computation. Indeed,
% for any convex function $\varphi \colon \mathbb{R}_+ \to \mathbb{R}$ with $\varphi(0)=0$, 
%     $$
%     \int \varphi(g) = \int \varphi(f (T) |T'|) \leq \int \varphi(f (T)) |T'| = \int \varphi(f)
%     $$
%     where the first equality follows from the equation $g(x) = f(T(x))|T'(x)|$, the inequality is a consequence of the fact that $t \mapsto \varphi(t)/t$ is non-decreasing and $|T'| \leq 1$, and the last equality is a change of variable. 
% \end{remark}
% }

\begin{remark} \label{rem:expsmall}
In the application we will use the lemma only when $A' \geq 1$ (\textit{i.e.}\ Lemma \ref{lem-intro: King Transport-Majorization}), which specifying to $\phi(x)=|x|^s$ leads to the family of inequalities 
$$
\int g^s d\mu \leq \int f^s d\nu , \qquad s \geq 1 .
$$
Note however that the lemma could potentially lead to a much stronger result in the situation where $A' > 1$. Indeed, for $\phi(x)=|x|^s$ the second conclusion of the lemma reads
$$
\int g^s d\mu \leq {A'}^{-(s-1)}\int f^s d\nu
$$
\textit{i.e.}\ there is an extra exponentially small factor. 
\end{remark}

\begin{remark} 
Observe that, if one looks for integral comparison, one can separate variables in, say, the conclusion $\int \varphi(g)d\mu \geq A \int \varphi(f/A)d\nu$, when $\varphi \geq 0$, by considering
$w(x)\coloneqq \sup_{u>0} \varphi(ux)/\varphi(u)$. Indeed, by definition of $w$ it holds
$$
\int \varphi(g)d\mu \geq  A \int \varphi(f/A)d\nu \geq \frac{A}{w(A)} \int \varphi(f) d\nu.
$$
Also, for $\varphi \geq 0$ satisfying the following so-called $\Delta_2$-condition (see \textit{e.g.}\ \cite{rao-ren}) $\varphi(2x) \leq K \varphi( x)$, for all $x >0$ and some $K \geq 2$, it holds
for $A >1$, 
$$
A \varphi\left (\frac{f}{A} \right) 
= 
K^{\log_K(A)} \varphi\left (\frac{f}{A} \right) 
\geq 
K^{\lfloor \log_K(A) \rfloor} \varphi\left (\frac{f}{A} \right)
\geq 
 \varphi \left(\frac{2^{\lfloor \log_K(A) \rfloor}}{A} f\right) 
$$
(where the floor signs denotes the entire part and $\log_K$ the logarithm in base $K$). As a conclusion we get that $\frac{2^{\lfloor \log_K(A) \rfloor}}{A} f \# \mu \prec_0 g \# \nu$. For the conclusion of the lemma involving $A'$, one needs to consider instead the $\nabla_2$-condition: $\varphi(x) \geq \frac{1}{2 \ell} \varphi(\ell x)$, for all $x >0$ and some $\ell >1$ \cite{rao-ren}.
\end{remark}

\begin{proof}[Proof of Lemma \ref{lem: King Transport-Majorization vanish}]
Define ${\mu_0}  = \mu/ A$, then ${G_0}(\lambda) \coloneqq {\mu_0}(g > \lambda) = \mu (g > \lambda)/ A = G(\lambda) / A$.  Further, define $\tilde{f} \coloneqq f/A$, $\tilde{F}(\lambda) = \nu( \tilde{f} > \lambda) = \nu ( f > \lambda A)= F(\lambda A)$, then
\begin{align*}
    \int [ g - t]^+ d \mu 
        &=
            \int_{\{ u > 0 \}} \left[ \frac{f(T)v(T)|T'|}{u} - t\right]^+ d\mu 
                \\
        &\geq 
            \int_{\{u > 0\}}
                {\left[ f(T(x)) - t A \right]^+ v(T(x))|T'(x)| dx}
                    \\
        &= 
            \int [ f - t A ]^+ d\nu
                \\
\end{align*}
where the inequality follows from the fact that $\left[\frac x y - t \right]_+ \geq [ x - t \bar y ]_+/y$ when $y \leq \bar y$ holds for $x,y,t \geq 0$, and the first equality follows a change of variables after observing that $\left[ f(T) - t A \right]^+ v(T)|T'| > 0 $ implies  $u > 0$.  Observe that since $\int [ f - t A ]^+ d\nu = A \int_t^\infty \tilde{F}(\lambda) d \lambda$ the above inequality can be re-written as
\begin{align*}
    \int_t^\infty G_0(\lambda) d\lambda \geq \int_t^\infty \tilde{F}(\lambda) d \lambda.
\end{align*}
Thus by Theorem \ref{thm: chong},
\begin{align*}
    \int \phi(g) d \mu_0 \geq \int \phi(\tilde{f}) d \nu
    %,
\end{align*}
and the result follows.
If $A \leq 1$, then since $\phi(0) = 0$, $\phi(t x) \geq t \phi(x)$ for $t \geq 1$ and hence
\begin{align*}
    \int A \varphi(f/A) d\nu \geq \int \varphi(f) d\nu,
\end{align*}
and the majorization follows from Theorem \ref{th: Chong vanish}.  

The argument for $A'$ is similar and left to the reader.
\end{proof}

% Specifying the conclusion of Lemma \ref{lem: King Transport-Majorization vanish}
% to power function{s} leads to Lemma \ref{lem-intro: transport} in the introduction that we now prove.

% \begin{proof}[Proof of Lemma \ref{lem-intro: transport}]
% Take $u \equiv v \equiv 1$.
% The assumptions of Lemma \ref{lem-intro: transport} guarantee that
% the constant $A'$ of Lemma \ref{lem: King Transport-Majorization} is greater than one, and therefore that $\int \varphi(g^{s_o}) d\mu \leq \varphi(f^{s_o}) d\nu$ for any convex function $\varphi \colon \mathbb{R}^+ \to \mathbb{R}$ with $\varphi(0)=0$. The desired conclusion follows choosing $\varphi(x)=x^{s/s_o}$.
% \end{proof}

{
\begin{remark} \label{rem:tm}
In short, Lemma \ref{lem: King Transport-Majorization vanish} demonstrates that \emph{transportation implies majorization}.  More explicitly, consider the case $u \equiv v \equiv 1$. If there exists a change of variable $T: \mathbb{R}^n \to \mathbb{R}^n$ from $g$ to $f$ (therefore such that $g(x) = f(T(x))|T'|$ for almost all $x$), and an $\varepsilon >0$ such that $|T'(x)| \geq \varepsilon$ holds almost surely, then $\tilde{T}(x) = T(x)/\varepsilon^{\frac 1 n}$ is a change of variable from $\tilde{g}(x) \coloneqq g(x)/\varepsilon$ to $\tilde{f}(x) = f( x \varepsilon^{\frac 1 n})$ (almost surely 
$    \tilde{g}(x) = \tilde{f}({\tilde{T}}(x)) |\tilde{T}'(x)|$),
with $|\tilde{T}'(x)|\geq 1$.  Thus $\tilde{f} \# \nu \prec_0 \tilde{g} \# \mu$.
\end{remark}
}

  In the following corollary we apply this observation to the convex function $\varphi(x) = x^s$.

\begin{cor} \label{cor:NPbis}
    For $f,g : \mathbb{R}^n \to \mathbb{R}_+$ integrable such that {there exists a change of vaiables $T$ from $g$ to $f$}
    %$g(x) = f(T(x)) |T'(x)|$, for $T: \mathbb{R}^n \to \mathbb{R}^n$ locally Lipschitz 
    with $\mathcal{T} \coloneqq \inf_x |T'(x)| > 0$, then
    \begin{align*}
           \int f^s(x) dx \leq \mathcal{T}^{1-s} \int g^s(x) dx , \qquad s \geq 1 .
    \end{align*}
\end{cor}

\begin{remark}
Written in terms of $L_p$ norms, the conclusion of the corollary becomes
\begin{align*}
    \| f \|_s \leq \mathcal{T}^{-\frac 1 {s'}} \|g\|_s
\end{align*}
% where $\frac 1 s + \frac 1 {s'} = 1$. { As already mentioned in Remark \ref{rem:expsmall}, f}or applications one can try to take advantage of the potentially small exponential factor $\mathcal{T}^{-\frac 1 {s'}}$.

Note that the statement is given for $s_o=1$ (we assumed $\int f{(x)dx}= \int g{(x)dx}$), where $s_o$ is as in Lemma \ref{lem-intro: transport}. Similar conclusion could be stated for any $s_o >0$.
%Note that $f(x) = g(T(x))|T'(x)|$ implies $\int f = \int g$ so that, $\int \varphi(f) \geq \int \varphi(g)$ for all convex functions satisfying $\varphi(0) = 0$ by Theorem \ref{thm: chong} item \eqref{item: all convex}.
\end{remark}

\begin{proof}
For $\tilde{g}(x) = g(x)/\mathcal{T}$, $\tilde{f}(x) = f(x \mathcal{T}^{\frac 1 n})$, and $\tilde{T}(x) = T(x)/\mathcal{T}^{\frac 1 n}$, by $g(x) = f(T(x))|T'(x)|$, it holds
$\tilde{g}(x) = \tilde{f}(\tilde{T}(x)) |\tilde{T}'(x)|$ 
and $|\tilde{T}'(x)| \geq 1$.
Thus, by Lemma \ref{lem: King Transport-Majorization vanish}, $\tilde{f} \# d\nu \prec_0 \tilde{g} \# \mu $ and applying $\int \varphi(\bar{f}) \leq \int \varphi(\bar{g})$ to the function $\varphi(x) = x^s$, we get
$\int \tilde{f}^s(x) dx \leq \int \tilde{g}^s(x) dx$.
This leads to the desired conclusion.
%    \int f^s(x) dx \leq \mathcal{T}^{ s} \int g^s(x \mathcal{T}^{\frac 1 n}) dx
%        \\
%    \leq \mathcal{T}^{ s - 1} \int g^s(x) dx,
%\end{align*}
%or
%\begin{align*}
%    \| f \|_s \leq \mathcal{T}^{\frac 1 {s'}} \| g \|_s
%\end{align*}
\end{proof}

We end this section with an alternative direct proof of Corollary \ref{cor:NPbis}, based on transport arguments.

\begin{lem} 
Let $\mu$ and $\nu$ be measures on $\mathbb{R}^n$ (not necessarily finite), such that $\mu(dx) = u(x) dx$ and $\nu(dx) = v(x) dx$, $f,g : \mathbb{R}^n \to \mathbb{R}_+$ be non-negative and { measruable} such that $\int f(x) v(x) < \infty$, and 
{assume that there exists a change of variables $T: \mathbb{R}^n \to \mathbb{R}^n$
from $gu$ to $fv$.}
%let $T: \mathbb{R}^n \to \mathbb{R}^n$ {be} locally Lipschitz, such that $g u = f(T) v(T) |T'|$. 
Then for all $s \geq 1$, 
\begin{align*}
    \int {f^s d\nu \leq}  A^{s-1}  \int g^s d { \mu}
\qquad \mbox{if} \quad 
    A \coloneqq \sup \frac{u}{v(T)|T'|} \in (0,\infty).
\end{align*}
Also, {if $T$ is invertible}
\begin{align*}
    \int g^s d\mu \leq  {\frac{1}{{A'}^{s-1}}}  \int f^s d \nu
\qquad \mbox{if} \quad 
    A' \coloneqq \inf \frac{u}{v(T)|T'|} \in (0,\infty) .
\end{align*}
\end{lem}

{
\begin{remark}
As earlier mentioned, if $T$ is the Brenier map, then it is invertible in our setting and therefore the second part of the Lemma applies.
\end{remark}
}

\begin{proof}
{Since $T$ is a change of variables from $ug$ to $fv$,
%The Monge-Amp\`ere equation $g u = f(T) v(T) |T'|$ guarantees that, 
}for any measurable function $h$,  $\int h(T) g d\mu = \int h f d\nu$. Applying this to $h=f^{s-1}$ it follows that
\begin{align*}
    \int f^s d\nu 
     & =
    \int f^{s-1} f d\nu \\
     & =
    \int f^{s-1}(T)gd\mu \\
    & \leq 
    \left( \int f^{s}(T) d\mu \right)^\frac{s-1}{s} \left( \int g^s d\mu \right)^\frac{1}{s} {,}
\end{align*}
where the last inequality follows from H\"older's inequality applied with $\frac{s-1}{s} + \frac{1}{s}=1$.

{By definition of $A$ and c}hanging variables  it holds
$\int f^{s}(T) d\mu \leq A \int f^{s}(T)v(T) |T'|  = {A} \int f^s {d\nu}$
from which the expected result follows.

The part with $A'$ is similar and left to the reader.
\end{proof}

%begin{remark} \label{rem: equality cases}
%The proof gives the following information. If %$\int f^s d\nu = \int g^sd\mu$ for some $s>1$, %under the hypotheses of the Lemma, then $f$ and %$g$ are almost surely constant functions %(recall that the Monge-Amp\`ere equation %already implies $\int f d\nu = \int g d\mu$ so %that $f$ and $g$ must satisfy two constraints). %To see this one only needs to keep track of %equality cases in the proof of the Lemma. %Details are left to the reader.
%\end{remark}

}

%%%%%%%%%%%%%%%%%%%%%%%%%%%%%%%%%%%%%%%%%%%%%%%%%%%%%%%%%%%%%%%%%%%%%%%%%%%%%%%%%%%%%%%%
%%%%%%%%%%%%%%%%%%%%%%%%%%%%%%%%%%%%%%%%%%%%%%%%%%%%%%%%%%%%%%%%%%%%%%%%%%%%%%%%%%%%%%%%

\section{Ball's integral inequality and beyond}

The aim of this section is to prove Theorem \ref{thm: discrete Ball's inequality}. As a warm up, and to show the efficiency of the transport approach of Lemma \ref{lem-intro: transport} (and Lemma \ref{lem: King Transport-Majorization vanish}), we may first reprove Ball's integral inequality and a 2-dimensional analog due to Oleskiewicz and Pe{\l}czy\'nski \cite{OP00}. Both proofs are very short and elementary.

\subsection{Ball's integral inequality}
Recall that Ball's integral inequality asserts that
\begin{equation*} 
\int_{-\infty}^\infty g(x)^{s} dx < \int_{-\infty}^\infty f(x)^s dx , \qquad s >1 
\end{equation*}
with
$$
f(x) := e^{-\pi x^2} \qquad \mbox{and} \qquad g(x):=  \left(\frac{\sin(\pi x)}{\pi x}\right)^2 , \qquad x \in \mathbb{R} .
$$
Ball's original proof is based on series expansion. As already mentioned, Nazarov and Podkorytov gave a very elegant and simple alternative proof of the latter, using Lemma \ref{lem:NP}. In this section, we present yet another proof, very elementary, based on Lemma \ref{lem-intro: transport}.

Since $\int f(x)dx=\int g(x)dx=1$, $f$ and $g$ are probability densities on the line. In that case, the transport map $T$ that pushes forward the probability measure with density $f$ onto that of density $g$ is increasing and given by
$T:=F^{-1} \circ G$ with
$$
F(x):= \int_{-\infty}^x f(t)dt , \qquad G(x):= \int_{-\infty}^x g(t)dt , \qquad  x \in \mathbb{R} .
$$
 The transport map $T$  { is a change of variables from $g$ to $f$}. Note that it is one-to-one increasing on $\mathbb{R}$ and that it satisfies by construction the Monge-Amp\`ere Equation $g=f(T)T'$.
 
 Now Ball's integral inequality will follow from Lemma \ref{lem-intro: transport} if we can prove that $T' \leq 1$, which is the aim of the next lemma.

\begin{lem} \label{lem:T'}
For all $x \in \mathbb{R}$, $T'(x) \leq 1$.
\end{lem}

\begin{proof}
Observe that, $T'\leq 1$ on $\mathbb{R}$ is equivalent\footnote{As a curiosity, in other contexts, the expression $F' \circ F^{-1}$ appears to be the isoperimetric profile associated to the probability measure with density $f$, and similarly for $g$. Therefore, the lemma asks for a comparison between two isoperimetric profiles.} to saying that $G' \circ G^{-1} \leq F' \circ F^{-1}$, and so to $g \circ G^{-1} \leq f \circ F^{-1}$ on $(0,1)$. Since $f$ and $g$ are even, $g \circ G^{-1}$ and $f \circ F^{-1}$ are symmetric about $1/2$. Therefore one needs to prove the inequality on $(1/2,1)$ only. The density $f$ being decreasing on $\mathbb{R}_+$ with inverse $f^{-1}(y)= \sqrt{\frac{1}{\pi} \log \left( \frac{1}{y}\right)}$ and $F$ being increasing, the inequality
$g \circ G^{-1} \leq f \circ F^{-1}$ on $(1/2,1)$ is in turn equivalent to $G \leq F \circ f^{-1} \circ g $ on $(0,\infty)$. This can be recast as
\begin{align} \label{eq:forJames}
\int_{-\infty}^x \left(\frac{\sin(\pi u)}{\pi u}\right)^2 du \leq \int_{-\infty}^{\sqrt{\frac{2}{\pi} \log \left( \left|\frac{\pi x}{\sin(\pi x)}\right| \right)}} e^{-\pi u^2} du,
\end{align}
for $x >0$.
For $x \in (0,1)$ we mimic an argument borowed from \cite{NP00}. Using the infinite product representation of the sinus,  for $x \in (0,1)$, on one hand one has
\begin{align*}
\left(\frac{\sin(\pi x)}{\pi x}\right)^2 
& = 
\prod_{k=1}^\infty \left(1 - \frac{x^2}{k^2} \right)^2 \\
& \leq 
\prod_{k=1}^\infty e^{-\frac{2x^2}{k^2}} \\
& = e^{-\frac{\pi^2 x^2}{3}} \\
& \leq 
e^{-\pi x^2} .
\end{align*}
This implies $\sqrt{\frac{2}{\pi} \log \left( \left|\frac{\pi x}{\sin(\pi x)}\right| \right)} \geq x$ for any $x \in (0,1)$ and therefore Inequality \eqref{eq:forJames} holds for any $x \in [0,1]$.
{
For $x > 1$, we reformulate \eqref{eq:forJames} as
\begin{equation} \label{eq:forJames2}
\int_{\sqrt{\frac{2}{\pi} \log \left( \left|\frac{\pi x}{\sin(\pi x)}\right| \right)}}^{\infty} e^{-\pi u^2} du
\leq 
\int_x^{\infty} \left(\frac{\sin(\pi u)}{\pi u}\right)^2 du , \qquad x >0,
\end{equation} 
and observe that} for $y=\sqrt{\frac{2}{\pi} \log \left( \left|\frac{\pi x}{\sin(\pi x)}\right| \right)}$, it holds
\begin{align*}
        \int_{y}^\infty e^{-\pi u^2} du
            & \leq
                \int_{y}^\infty \frac{2\pi u}{2\pi y} e^{-\pi u^2} du \\
            & =
                \frac{e^{-\pi y^2}}{2\pi y} \\
            & = 
            \frac {\left( \frac{\sin( \pi x)}{\pi x} \right)^2}{2\sqrt{ 2\pi \log \left( \left|\frac{\pi x}{\sin(\pi x)}\right| \right)}} 
            %\leq 
            %\frac {1/(\pi x)^2}{2\sqrt{ 2\pi \log \left( \left|\frac{\pi x}{\sin(\pi x)}\right| \right)}} 
            .
    \end{align*}
%Meanwhile for an integer $k \geq 2$ and any $\varepsilon \in [0,1/2]$
%    \begin{align*}
%        \int_{k}^{k+1} \left(\frac{\sin (\pi u)}{\pi u} \right)^2 du 
%            \geq 
%        \int_{k+\varepsilon}^{k+1-\varepsilon} \left(\frac{\sin (\pi u)}{\pi u} \right)^2 du
%        \geq 
%        \frac{\sin(\pi \varepsilon)^2}{\pi^2} \int_{k+\varepsilon}^{k+1-\varepsilon} \frac{1}{u^2}du
%            =
%           %\frac{\sin(\pi \varepsilon)^2}{\pi^2}
%    \end{align*}
Meanwhile for an integer $k \geq 2$,
    \begin{align*}
        \int_{k}^{k+1} \left(\frac{\sin (\pi u)}{\pi u} \right)^2 du 
           & \geq 
        \frac 1 {(k+1)^2 \pi^2} \int_{k}^{k+1} \sin^2(\pi u) du \\
           & =
            \frac{1}{2\pi^2} \frac{1}{(k+1)^2} .
    \end{align*}
    For comparison, $\int_{k}^{k+1} \frac{1}{u^2} du  = \frac 1 {k (k+1)}$ so that, for $k \geq 2$
        \begin{align*}
            \int_{k}^{k+1} \left(\frac{\sin (\pi u)}{\pi u} \right)^2 du 
               & \geq
                    \frac{k}{2\pi^2(k+1)} \int_{k}^{k+1} \frac{1}{u^2} du  \\
  %              = \frac{1}{2\pi^2} u_k ( v_k - v_{k+1})    
               & \geq
                    \frac 1 {3\pi^2} \int_{k}^{k+1} \frac{1}{u^2} du .
        \end{align*}
%with $u_k= k/(k+1)$ and $v_k=\int_k^\infty \frac{1}{u^2} du$
    Put $\lceil x \rceil = \inf_{k \in \mathbb{Z}} \{ k : k  \geq x \}$ for the ceiling part of $x$ that is greater or equal to 2, since $x>1$.    We have
    \begin{align*}
        \int_{x}^{\infty} \left(\frac{\sin (\pi u)}{\pi u} \right)^2 du 
          &  \geq 
         \frac 1 {3\pi^2} \int_{\lceil x \rceil}^{\infty} \frac{1}{u^2} du   \\
         & = \frac 1 {3 \pi^2 \lceil x \rceil} .
    \end{align*} 
We claim that 
$$
\frac {(\sin(\pi x))^2/\left( \pi x \right)^2}{2\sqrt{ 2\pi \log \left( \left|\frac{\pi x}{\sin(\pi x)}\right| \right)}}  \leq \frac 1 {3 \pi^2 \lceil x \rceil}, \qquad \forall x >1
$$ 
from which  \eqref{eq:forJames2} for $x>1$ immediately follows. Therefore, to complete the proof of the lemma, it only remains to  prove the claim. We proceed to further reductions. Squaring and exploiting the fact that, for $x>1$, $\frac{\lceil x \rceil}{x} \leq 2$ and $|\sin(\pi x)| \leq 1$, the claim will in fact be a consequence of 
$$
9\pi\leq y \log y , \qquad y>20
$$
where we changed variable ($y =\left|\frac{\pi x}{\sin(\pi x)}\right|^2$ whose minimum is achieved on the interval $[\pi,2\pi]$ and is greater than $20$).
The claim immediately follows and this ends the proof of the lemma.
\end{proof}

\begin{remark}
In \cite[Proposition 7]{konig-koldobsky} the authors prove, using Nazarov and Podkorytov's lemma, the following inequality which is a refined version of Ball's integral inequality
\begin{equation*} 
\int_{-\infty}^\infty      \left| \frac{\sin (\pi u)}{\pi u} \right|^{2s} du \leq \sqrt{\frac{3}{\pi}} \int_{-\infty}^\infty e^{-\pi s u^2} du, \qquad s \geq \frac{9}{8} .
\end{equation*}
The refinement is coming from the factor $\sqrt{3/\pi} < 1$.
We mention that Lemma \ref{lem-intro: transport} can also be applied to prove such an inequality with the exact same line of argument as above.
\end{remark}

%%%%%%%%%%%%%%%%%%%%%%%%%%%%%%%%%%%%%%%%%%%%%%%%%%%%%%%%%%%%%%%%%%%%%%%%

\subsection{Oleszkiewicz-{Pe{\l}czy{\'n}ski}: a 2-dimensional Ball's type integral inequality.}

For $v \geq 1$, let $j_v(s) = 2^v \Gamma(v + 1) J_v(s)/s^v$ where $J_v(s) = \sum_{m=0}^\infty \frac{(-1)^m}{m! \Gamma(m + v +1)} \left( \frac s 2 \right)^{2m + v}$ is the Bessel function of order $\nu$.
For $p \geq 2$ and $m \geq 2$ integer, consider the following integral inequality
\begin{equation} \label{eq:brzezinski}
    \int_0^\infty |j_{\frac m 2}(s)|^p s^{m-1} ds \leq \int_0^\infty \left(e^{-\frac{s^2}{2m+4}} \right)^p s^{m-1} ds.
\end{equation}

%This can be understood as an $m$-dimensional generalization of Ball's inequality.  It was proved in \cite{brzezinski2013volume} for $m \geq 6$ or $7$ and for a slightly smaller restriction on $p$. This paper is rather technical and hard to read in my opinion.  For me it is not even easy to check that this result is proven in the paper.
The case  $m=1$, as observed in \cite[Remark 4.3]{OP00}, reduces to Ball's integral inequality since
$J_{1/2}(t)= \left( \frac{2}{\pi} \right)^{1/2} \frac{\sin t}{t}$ (see \cite[Page 54 Inequality (3)]{watson}) and therefore $j_{1/2}(s)=\frac{\sin s}{s}$.

The case $m=2$, treated below with the help of our transport Lemma \ref{lem: King Transport-Majorization vanish},  was originally proved in \cite{OP00} { through careful and direct analysis, using detailed expansions and approximations. Yet another proof can be found in \cite{dirksen} {which} relied on Nazarov-Podkorytov's lemma} (to be complete, H. Dirksen mentions the existence of an unpublished note by K\"onig that inspired him and that uses Nazarov-Podkorytov's lemma).

We refer the reader to Remark \ref{rem:brzezinski} below for more comments on 
Inequality \eqref{eq:brzezinski} for $m \geq 3$ and related inequalities.

For now, we fix $m=2$. 
Set $s=p/2$, $f(x)=e^{-x^2/4}$ and $g(x)=\left( \frac{2J_1(x)}{x} \right)^2$, $x >0$. Set $\mu$ for the measure on $(0,\infty)$ with density $x$ with respect to the Lebesgue measure. Then {the} inequality we want to prove takes the form
$$
\int_0^\infty  g^s d\mu  
\leq 
\int_0^\infty f^s d\mu,
\qquad s \geq 1 .
$$ 
Observe that $f$ and $g$ have same mass: $\int f d\mu= \int g d\mu= 2$. 
Let $T \coloneqq F^{-1} \circ G$ with, for $x >0$, 
\begin{equation} \label{eq:explicit}
F(x) \coloneqq \int_0^x f(t) \mu(dt) = 2(1-e^{-x^2/4})
, \qquad
G(x)\coloneqq \int_0^x g(t) \mu(dt)= 2 - 2(J_1^2(x)+J_0^2(x))
\end{equation}
(see Lemma \ref{lem:integral-of-g} below for the computation of $G$).
By construction, $T$ is the (increasing) pushforward of $g\mu$ onto $f\mu$, $T\# (g\mu) = f\mu$ and satisfies the Monge-Amp\`ere equation $g(x)x=f(T(x))T(x)|T'(x)|$.
Therefore, by Lemma \ref{lem: King Transport-Majorization vanish} (with $u(x)=v(x)=x$), the desired inequality would follow if we can prove that $T(x)T'(x) \leq x$.
This is achieved in the next lemma.

\begin{lem}
For all $x >0$ it holds  $T(x)T'(x) \leq x$.
\end{lem}

\begin{proof}
We distinguish between two cases.\\
$\bullet$ For  $x \in (0,2)$ we prove first that $g(x) \leq f(x)$. Expanding, this is equivalent to proving that (note that $J_1 \geq 0$ on $[0,2]$)
\begin{align*}
xe^{-x^2/8}-2J_1(x) = \sum_{n=0}^\infty (-1)^n \frac{1}{n!}\left( \frac{1}{2^{n-1}} - \frac{1}{(n+1)!} \right) \left( \frac{x}{2} \right)^{{2n+1}} \geq 0 , \qquad 0 \leq x \leq 2 .
\end{align*}
%Observe that the first two terms ({for} $n=0$ and $n=1$) vanish. 
We set, for $n \geq { 0}$,
$u_n := \frac{1}{n!}\left( \frac{1}{2^{n-1}} - \frac{1}{(n+1)!} \right) \left( \frac{{x}}{2} \right)^{{2n+1}}$ so that $xe^{-x^2/8}-2J_1(x) = \sum_{n={ 0}}^\infty (-1)^nu_n$ is an alternating series ($u_n > 0$ for $n \geq { 0})$. Now for $x \in [0,2]$, it is easy to check that $(u_n)_n$ is decreasing. In particular the alternating series has the sign of its first term {$u_0$}, which is positive, proving the claim. 

Now $g \leq f$ on $[0,2]$ implies that $G \leq F$ and therefore that $T(x) \leq x$ on $[0,2]$. In particular, the claimed inequality $T(x)T'(x) \leq x$ would be a consequence of $T' \leq 1$, or $g \leq f(T)$. Since $g \leq f$ this is implied by $f(T) \geq f$, which holds since $f$ is decreasing and $T(x) \leq x$.

{$\bullet$ For $x \geq 2$, since $T'(x)T(x)/x=g(x)/f(T(x))$ by the Monge-Amp\`ere Equation, the thesis $T(x)T'(x) \leq x$ is equivalent to saying that $g(x) \leq f(T(x))$, $x \geq 2$. 
% The function $f$ being decreasing and $F$ increasing, this is, in turn, equivalent to $G(x) \leq F \left( f^{-1} \left(g(x)\right) \right)$, $ x\geq 2$. Finally, 
Using the explicit expressions of $F$ and $G$ given in \eqref{eq:explicit}, this amounts to proving that 
$$
g(x) \leq J_1^2(x) + J_0^2(x), \qquad x  \geq 2 .
$$
This trivially holds, since for $x \geq 2$, 
$g(x) = \frac{4}{x^2}J_1^2(x) \leq J_1^2(x)$. }
\end{proof}

%Follow some computations related to the map $g$.

\begin{lem} \label{lem:integral-of-g}
For all $x \geq 0$, it holds
$$
\int_0^x g(t)tdt = 2 -2 (J_1^2(x) + J_0^2(x)) .
$$
\end{lem}

\begin{proof}
Observe that $J_1=-J_0'$ and that $J_0$ is satisfying the following equation 
$J_0''(t)+J_0(t) =J_0'(t)/t$. Therefore
$$
\frac{J_1^2(t)}{t}= \frac{J_0'(t)^2}{t} = - J_0''(t)J_0'(t) - J_0'(t)J_0(t) 
$$
from which the result follows by integration.
\end{proof}

\begin{remark} \label{rem:brzezinski}
We comment on Inequality \eqref{eq:brzezinski}. 
For $m \geq 3$, observe that $\int_0^\infty e^{-y^2/2} y^{m-1} dy= 2^\frac{m-2}{2}\Gamma(m/2)$. Therefore,
after change of variable
$$
 \int_0^\infty \left(e^{-\frac{s^2}{2m+4}} \right)^p s^{m-1} ds
 = 
\left(\frac{m+2}{p}\right)^\frac{m}{2} \int_0^\infty e^{-\frac{x^2}{2}} x^{m-1} dx
=
\frac{1}{2}\Gamma(m/2) \left(\frac{2m+4}{p}\right)^\frac{m}{2} .
$$
In turn, Inequality \eqref{eq:brzezinski} {{can}} be recast as
$$
\int_0^\infty |j_{\frac m 2}(s)|^p s^{m-1} ds \leq \frac{1}{2}\Gamma(m/2) \left(\frac{2m+4}{p}\right)^\frac{m}{2}.
$$
Such an inequality was proved by Brzezinski, \cite[Lemma 3.5]{brzezinski2013volume}, 
for $m \geq 3$ integer and $p > p_o >2$ for some well defined $p_o$. 
His proof relies on Nazarov-Podkorytov's lemma.

Let us mention two other inequalities of similar type.
In \cite[Inequality (10)]{OP00} the authors mention the following one
$$
 \int_0^\infty |j_{\frac m 2}(s)|^p s^{m-1} ds 
 \leq 
 \left( \int_0^\infty |j_{\frac m 2}(s)|^2 s^{m-1} ds \right) \frac{2^{m/2}}{p^{m/2}}, \qquad p \geq 2 . 
$$
They suggest that this might hold iff $1 \leq m \leq 2$. They also mention that "K\"onig noticed that it is false for $m=3,4,\dots$"...

On the other hand, Dirksen \cite[Theorem 4]{dirksen} proved the following inequality
$$
 \int_0^\infty |j_{\frac m 2}(s)|^p s^{m-1} ds \leq \sqrt{\frac{\pi(m+2)}{2p}} , \qquad p \geq 2, \quad m \geq 2 \mbox{ integer}.
$$
It should be clear that the latter is different from \eqref{eq:brzezinski}. In fact, it is equivalent to saying that
$$
 \int_0^\infty |j_{\frac m 2}(s)|^p s^{m-1} ds \leq \int_0^\infty \left(e^{-\frac{x^2}{m+2}}\right)^p dx , \qquad p \geq 2, \quad m \geq 2 \mbox{ integer}.
$$
The difference between the latter and \eqref{eq:brzezinski} is coming from {the absence of} the factor $s^{m-1}$ in the integral in {the} right hand side. Dirksen's proof uses
Nazarov-Podkorytov's lemma (for $m=2$ and $m \geq 5$), the case $m=3,4$ uses the techniques of \cite{OP00}).

The above inequalities are related to convex geometry (slices of cylinders, volume estimates etc.). Their proofs are quite involved and we hope that the techniques developed in Lemma \ref{lem: King Transport-Majorization vanish} could help for smoother approaches.
\end{remark}

%%%%%%%%%%%%%%%%%%%%%%%%%%%%%%%%%%%%%%%%%%%%%%%%%%%%%%%%%%%%%%%%%%%%%%%%%%%%
%%%%%%%%%%%%%%%%%%%%%%%%%%%%%%%%%%%%%%%%%%%%%%%%%%%%%%%%%%%%%%%%%%%%%%%%%%%%

\subsection{Discrete analog of Ball's integral inequality: proof of Theorem \ref{thm: discrete Ball's inequality}}\label{sec:transportBall}

We recast the statement of Theorem \ref{thm: discrete Ball's inequality}, which can be considered a discrete analog of Ball's integral inequality, as $L^p$-norm comparison.  That is, for $p \geq 2$, and $2 \leq n \in \mathbb{N}$,
\begin{align} \label{eq:discreteBall}
        \int_{-\frac 1 2}^{\frac 1 2} \left|\frac{\sin( n \pi x)}{n \sin \pi x} \right|^p dx < \int_{-\infty}^\infty \left(e^{- \pi (n^2 - 1) x^2/2}\right)^p dx  = \sqrt{\frac{2}{p(n^2-1)}} .
\end{align}

The proof of the theorem uses Lemma \ref{lem-intro: transport}. Before moving to it, let us explain why \eqref{eq:discreteBall} is stronger than Ball's integral inequality and give a Corollary on discrete slicing.

Using the substitution $u = n x$, and $|\sin x| \leq |x|$, it holds
\begin{align*}
    \int_{-1/2}^{1/2} \left| \frac{\sin(n \pi x)}{n \sin(\pi x)} \right|^p dx = \int_{-n/2}^{n/2} \left| \frac{\sin(\pi x)}{n \sin(\pi x/n)} \right|^p \frac{dx}{n} \geq \int_{-n/2}^{n/2} \left| \frac{\sin(\pi x)}{\pi x} \right|^p \frac{dx}{n}
\end{align*}
Therefore, \eqref{eq:discreteBall} implies that
\begin{align*}
    \sqrt{\frac{n^2-1}{n^2}}\int_{-n/2}^{n/2} \left| \frac{\sin(\pi x)}{\pi x} \right|^p dx 
    \leq 
    \sqrt{\frac 2 p} 
    = 
    \int_{-\infty}^\infty e^{-p\pi x^2/2}dx,
\end{align*}
which yields Ball's inequality with $n \to \infty$.

\begin{cor}[Discrete slicing] \label{cor: discrete slicing}
For $k_i \in \mathbb{Z}$ and $1 \leq l_i \in \mathbb{Z}$, the rectangular subset of $\mathbb{Z}^n$, $L = \llbracket k_1, k_1 + l_1 -1\rrbracket \times \cdots \times \llbracket k_n, k_n + l_n -1 \rrbracket$ satisfies
\begin{align*}
    \# \left\{ z \in L : \sum_{i=1}^n z_i = k \right\} < \sqrt{2} \frac{\prod_{i=1}^n l_i}{\sqrt{\sum_{j=1}^n (l_j^2 -1)}}
\end{align*}
for any $k \in \mathbb{Z}$.
\end{cor}

\begin{proof}
Without loss of generality, let $k_i = 1$.  We will split the argument in two cases, for the first, suppose that there exists $l_{i'}$ such that 
\begin{align*}
    \sum_{j=1}^n (l_j^2 -1 ) < 2 (l_{i'}^2 - 1)
\end{align*}
In this case,
\begin{align*}
    \sqrt{2} \frac{\prod_{i=1}^n l_i}{\sqrt{\sum_{j=1}^n (l_j^2 -1)}}
        >
           \frac{ \prod_{j=1}^n l_j}{\sqrt{l_{i'}^2 - 1}}
        >
            \prod_{j \neq i'} l_j,
\end{align*}
which is clearly larger than $\# \left\{ z \in L : \sum_{i=1}^n z_i = k \right\}$, since for $m = \{m_j\}_{j \neq i'}$, $\{ z \in L : z_j = m_j \mbox{ for } j \neq i', \sum_l z_l = k \}$ has at most one element.

Now we assume $\sum_{j=1}^n (l_j^2 -1 ) \geq 2 (l_{j_o}^2 - 1)$ holds for all ${j_o}$.  Take $X_{j}$ to be independent and uniformly distributed on $\llbracket 1, l_{j} \rrbracket$ { and set $X = X_1 + \cdots + X_n$. Then, using the triangle inequality and then  H\"older's inequality with exponents $p_j$, $j=1,\dots,n$, satisfying $\sum \frac{1}{p_j} =1$, we get}
\begin{align}
   \frac{ \# \left\{ z \in L : \sum_{{j} =1 }^n z_{j} = k \right\} }{ \prod_{{j}=1}^n l_{j}}
        &=
            \mathbb{P}( X = k ) \nonumber \\
        & {  = \int_{-\frac 1 2}^{\frac 1 2}  \mathbb{E} e^{{2i\pi t (X-k)}}  dt
            }
            \nonumber
                \\
        &\leq
            \int_{-\frac 1 2}^{\frac 1 2} | \mathbb{E} e^{{2i\pi t (X-k)}} | dt \nonumber
                \\
        &\leq
            \prod_{{j}=1}^n \left( \int_{-\frac 1 2}^{\frac 1 2} |\mathbb{E} e^{{2i\pi} tX_{j}}|^{p_{j}} \right)^{\frac 1 {p_{j}}}. \label{eq: product of integrals}
\end{align}
Observing that {$|\mathbb{E} e^{{2i\pi} t X_j}| = \left|\frac{\sin( l_j \pi t)}{l_j \sin \pi t} \right|$}, we set $p_j = \frac{\sum_{i=1}^n {(l_{i}^2 - 1)}}{l_j^2 - 1} \geq 2$ and apply \eqref{eq:discreteBall} to \eqref{eq: product of integrals} to obtain
\begin{align*}
    \mathbb{P}(X = k) < \prod_{{j}=1}^n \left( \sqrt{\frac{2}{p_{j} (l_{j}^2-1)}} \right)^{\frac 1 {p_{j}}} = \sqrt{\frac{2}{\sum_{{i}=1}^n (l_{i}^2 -1)}},
\end{align*}
and the result follows.
\end{proof}

We note that equality can be obtained in following limit. Take $l_1 = l_2 = m$ and fixing $l_3, \dots, l_n =1$, then 
\begin{align*}
    \# \{ z \in L : \sum_{i=1}^n z_i = m + n-3\} = m
\end{align*}
while 
\begin{align*}
   \frac{ \prod_{j=1}^n l_j}{\sum_{i=1}^n (l_i^2 -1)} = \frac{m^2}{\sqrt{2(m^2-1)}}.
\end{align*}
Thus, the constant in Corollary \ref{cor: discrete slicing} cannot be improved, as
\begin{align*}
    \lim_{m \to \infty} \frac{\# \left\{ z \in L : \sum_{i=1}^n z_i = k \right\} }{\frac{\prod_{i=1}^n l_i}{\sqrt{\sum_{j=1}^n (l_j^2 -1)}}} = \sqrt{2}. 
\end{align*}

We now turn to the proof of Theorem \ref{thm: discrete Ball's inequality}.

\begin{proof}[Proof of Theorem \ref{thm: discrete Ball's inequality}]
The proof of Theorem \ref{thm: discrete Ball's inequality} relies on Lemma \ref{lem-intro: transport}.
Observe that the two relevant functions in \eqref{eq:discreteBall} do not have equal mass\footnote{{To see the first equality take a uniform random variable $X$ on $\{0,1,\dots,n-1\}$ and set $\varphi_X(t) = \mathbb{E} e^{2 \pi i X}$.  By Plancherel,
$\frac 1 n = \sum_{k=0}^{n-1} P(X=k)^2 = \int_0^1 | \varphi_X(t)|^2 dt = \int_{-1/2}^{1/2} | \varphi_X(t)|^2 dt$ by periodicity, while
$\varphi_X(t) 
=
\frac 1 n \sum_{k=0}^{n-1} \left(e^{2 \pi i t}\right)^k 
=
e^{\pi i t (n-1)} \frac{ \sin( \pi n t)}{n \sin(\pi t)}.
$}}:
\begin{align} \label{eq:hard}
\int_{-1/2}^{1/2} \left(\frac{\sin( n \pi x)}{n \sin(\pi x)} \right)^2 dx = \frac{1}{n}
\qquad 
\mbox{and}
\qquad 
\int_{-\infty}^{\infty} e^{-\pi(n^2 - 1) x^2} = \frac{1}{\sqrt{n^2-1}} .
\end{align}
Therefore, we define for $x \in [0,\infty)$,
\begin{align*}
    g(x) = \mathbbm{1}_{[0,\frac 1 2]}(x) \left(\frac{\sin(n \pi x)}{n \sin (\pi x)}\right)^2
    \qquad 
\mbox{and}
\qquad 
f(x) = \mathbbm{1}_{[0,A]}e^{-\pi (n^2-1) x^2},
\end{align*}
where $A$ is determined by the equation 
$$
\int_0^A e^{- \pi (n^2-1)x^2} dx = \frac{1}{2n} .
$$ 
Note that $\int_0^\infty e^{-\pi (n^2-1)x^2} dx = \frac 1 {2 \sqrt{n^2-1}} > \frac 1 {2n}$ so that $A$ is well defined. Further since $A$ is finite, upon completing the proof, we arrive at the strict inequality.  By construction $\int_0^\infty g(x)dx = \int_0^\infty f(x)dx=1/(2n)$.

Define $G: [0,1/2] \to [0, \frac 1 {2n}]$ as $G(x) = \int_0^x g(t) dt$ and $F: [0,A] \to [0,\frac{1}{2n}]$ by $F(x) = \int_0^x f(t) dt$. Put $T = F^{-1} \circ G$ that is, by construction, the pushforward of the measure with density $g$ onto that of density $f$, {which} satisfies the Monge-Amp\`ere equation $g=f(T)T'$ (observe that $T$ is increasing). Therefore, by Lemma \ref{lem-intro: transport}, the Theorem will follow if we can prove that $T' \leq 1$ which holds by Lemma \ref{lem:T'discrete}.
This achieves the proof of the Theorem.
\end{proof}

\begin{lem} \label{lem:T'discrete}
For all $x \in [0,1/2]$, $T'(x) \leq 1$.
\end{lem}

As a technical preparation, we observe that, 
\begin{equation} \label{eq:sinus}
\sin a \geq \frac{a}{b} \sin(b) , \qquad 0 < a \leq b \leq \pi/2, 
\end{equation}
a direct consequence of the fact that $x \mapsto \frac{\sin(x)}{x}$ is decreasing on $[0,\pi/2]$.

\begin{proof} [Proof of Lemma \ref{lem:T'discrete}]
We need to prove that  $g(x) \leq f(F^{-1}(G(x)))$ for any $x \in [0,1/2]$. Put 
$$
I:=\{x \in [0,1/2] : g(x) > f(A)\}.
$$
Then, for any $x \in [0,1/2] \setminus I$, it holds
$g(x) \leq f(A) \leq f(F^{-1}(G(x)))$. Therefore we only need to prove the desired inequality $g(x) \leq f(F^{-1}(G(x)))$ for $x \in I$. 

For $x \in I$, composing by $f^{-1}$, that is decreasing, and then by $F$, that is increasing, 
 $g(x) \leq f(F^{-1}(G(x)))$ is equivalent to $F( f^{-1} (g((x))) \geq G(x)$. Hence, we need to prove that
\begin{align} \label{eq:start}
    \int_0^x g(t) dt \leq \int_0^{f^{-1}(g(x))} f(t) dt,
\end{align}
for all $x \in I$.
To that aim, we need to distinguish between different regimes and to proceed to further successive reductions.

\bigskip

We first prove the inequality for $x \in [0,1/n] \cap I$.  For this we need only to prove $f(x) \geq g(x)$, as this will give $f^{-1}(g(x)) \geq x$ since $f$ is decreasing, and hence
\begin{align*}
    \int_0^x g(t) dt \leq \int_0^x f(t) dt \leq \int_0^{f^{-1}(g(x))} f(t) dt.
\end{align*}
The inequality  $g(x) \leq f(x)$ on $(0,1/n)$ is a direct consequence of the fact that, on $(0,1/n)$, $\frac{\sin(n \pi x)}{n \sin (\pi x)} \leq e^{-(n^2-1)\pi^2 x^2/6}$ that we now prove. By the product expansion for $\sin$,
\begin{align*}
    \frac{ \sin(n \pi x)}{n \sin(\pi x)} = \frac{ \frac{\sin(n \pi x)}{n \pi x}}{ \frac{\sin(\pi x)}{\pi x}} = \prod_{k=1}^\infty \frac{ 1- \left( \frac{nx}{k} \right)^2}{1- \left(\frac x k \right)^2},
\end{align*}
while the identity $\sum_{k=1}^\infty \frac 1 {k^2} = \frac{\pi^2}{6}$, gives 
\begin{align*}
    e^{-(n^2-1)\pi^2 x^2/6} = \prod_{k=1}^\infty e^{-\frac{(n^2-1) x^2}{k^2}}.
\end{align*}
Comparing term-wise, using
$
    e^{-\frac{(n^2-1) x^2}{k^2}} \geq  1- (n^2-1)\frac{x^2}{k^2},
$
and writing $ y = x^2 /k^2$ it suffices to prove
\begin{align}
    (1 - (n^2-1) y)(1-y) \geq 1-n^2 y
\end{align}
for $y \in (0, 1/n^2)$.  But this is equivalent to $(n^2-1)y^2 \geq0$ so the claim follows. 
%For $n \geq 5$, $\pi n^2/2 \leq (n^2-1) \pi^2/6$.\\

Note that the above argument shows that \eqref{eq:start} holds for $n=2$ for all $x \in I$. 
We therefore deal in the sequel with $n \geq 3$. One key ingredient is the following lemma whose proof is postponed 
%at 
{ to} the end of the section.

\begin{lem} \label{lem:James}
For $n \geq 3$, and $x \in [1/2 - 1/n, 1/2] \cap (1/n, 1/2]$ it holds
\begin{align*}
    f(A) \geq g(x) .
\end{align*}
\end{lem}

This lemma guarantees that, for $n=3$, $I \cap {(1/3,1/2]}=\emptyset$ and therefore \eqref{eq:start} is proved for $n=3$ and we can assume that $n \geq 4$.
The lemma also guarantees that, for any $n \geq 4$, 
$I \cap [1/2-1/n,1/2] = \emptyset$. Therefore, it only remains to prove
\eqref{eq:start} for $n \geq 4$ in the intermediate regime
$[1/n,1/2-1/n] \cap I$. Observe that $[1/n,1/2-1/n] = \{\frac{1}{4} \}$ for $n=4$ so that 
\eqref{eq:start} holds for $n=4$ also, by continuity and we are left with the regime  $[1/n,1/2-1/n] \cap I$ for $n \geq 5$.

Assume that $n \geq 5$ and consider the regime $[1/n,1/2-1/n] \cap I$.
Since by construction 
$$
\int_0^{1/2} g(t)dt = \int_0^A f(t)dt = \frac{1}{2n}
$$
Inequality \eqref{eq:start} in the studied regime is equivalent to the tail inequality
\begin{equation} \label{eq:forCyril}
\int_{f^{-1}(g(x))}^A f(t) dt 
\leq 
\int_x^{1/2} g(t) dt , \qquad x \in [1/n,1/2-1/n] .
\end{equation}

Observe that
\begin{align*}
    \int_{f^{-1}(g(x))}^A f(t) dt 
    & = 
    \int_{f^{-1}(g(x))}^A \frac{2\pi(n^2-1)t}{2\pi(n^2-1)t} e^{-\pi(n^2-1)t^2} dt \nonumber \\
    & \leq
    \frac{1}{2 \pi(n^2-1) f^{-1}(g(x))} \int_{f^{-1}(g(x))}^A 2\pi(n^2-1)t e^{-\pi(n^2-1)t^2} dt \nonumber \\
    & =
    \frac{g(x) - f(A)}{2 \pi(n^2-1) f^{-1}(g(x))} %\label{eq:f(A)} 
    \\
    & \leq 
    \frac{g(x)}{2 \pi(n^2-1) f^{-1}(g(x))}. \nonumber 
\end{align*}
Therefore, using the explicit expression for $f^{-1}(y)=\sqrt{\log(1/y)/(\pi(n^2-1))}$, \eqref{eq:forCyril}  would follow from
\begin{equation} \label{eq:pain}
\frac{g(x)}{2\sqrt{\pi(n^2-1) \log(1/g(x))}} 
\leq 
\int_x^{1/2} g(t)dt .
\end{equation}

In order to bound from below $1/g(x)$, we need to distinguish between $n=5$ and $n \geq 6$.

Let us first consider the case $n \geq 6$. We claim that, for $x  \in [1/n,1/2]$, $\frac{1}{g(x)} \geq 14$ (with the convention that $1/g(x)=\infty$ when $g(x)=0$).
Now fix $\theta \in (0,1/2)$.  If $\frac{1}{n} \leq x \leq \frac{1+\theta}{n}$, we have for $n \geq 6$
\begin{align*}
\frac{1}{g(x)} 
& = 
\frac{n^2 \sin(\pi x)^2}{\sin(n \pi x)^2}  \\
& \geq 
n^2 \frac{\sin(\pi/n)^2}{\sin(\theta \pi)^2} \qquad (\mbox{since } x \mapsto \sin^2 (x) \mbox{ is increasing on } (0, \pi/2) \cup (\pi,3\pi/2))\\
%& =
%\frac{\pi^2}{\sin(\theta \pi)^2}  \left( \frac{\sin(\pi/n)}{\pi/n}\right)^2 \\
& \geq 
\frac{36 \sin(\pi/6)^2}{\sin(\theta \pi)^2}\qquad (\mbox{by } \eqref{eq:sinus} \mbox{ with } a=\frac{\pi}{n} \mbox{ and } b= \frac{\pi}{6})\\
& = 
\frac{9}{\sin(\theta \pi)^2} 
\end{align*}
Similarly, for $x  \in [(1+\theta)/n,1/2]$,  it holds
\begin{align*}
\frac{1}{g(x)} 
& = 
\frac{n^2 \sin(\pi x)^2}{\sin(n \pi x)^2} 
\geq 
n^2 \sin^2\left(\frac{(1+\theta)\pi}{n} \right) 
%=
%\pi^2 (1+\theta)^2\left( \frac{\sin\left(\frac{(1+\theta)\pi}{n} \right)}{\frac{(1+\theta)\pi}{n}}\right)^2 \\
%& \geq 
%\pi^2 (1+\theta)^2\left( \frac{\sin\left(\frac{(1+\theta)\pi}{6} \right)}{\frac{(1+\theta)\pi}{6}}\right)^2
%=
\geq 
36 \sin^2\left(\frac{(1+\theta)\pi}{6} \right) .
\end{align*}
Therefore,
$$
\frac{1}{g(x)}  \geq 9 \max_{\theta \in (0,1/2)} \min \left( \frac{1}{\sin(\theta \pi)^2} , 4 \sin^2\left(\frac{(1+\theta)\pi}{6} \right)  \right) .
$$
For $\theta = 0.295$, we obtain $\min \left( \frac{1}{\sin(\theta \pi)^2} , 4 \sin^2\left(\frac{(1+\theta)\pi}{6} \right)  \right) = \frac{1}{\sin(\theta \pi)^2} \simeq 1.56$ from which we deduce that $1/g(x) \geq 14$ as announced, proving the claim.

It follows that, for any $n \geq 6$,
$$
2\sqrt{\pi(n^2-1)\log(1/g(x))} 
= 
2\sqrt{\pi\log(1/g(x))} \sqrt{\frac{n^2-1}{n^2}} n 
\geq  
2\sqrt{\pi \log(14)} \sqrt{\frac{35}{36}} n 
\geq 
5.67 n .
$$

In turn \eqref{eq:forCyril} would follow from
$$
g(x) \leq 5.67 n \int_x^{1/2} g(t)dt .
$$
Our aim is now to bound from below the right hand side of the latter. 
Using \eqref{eq:sinus} and a change of variables, it holds
\begin{align*}
    \int_x^{1/2} g(t)dt 
    & = \frac{1}{n^2 \sin(\pi x)^2} \int_x^{1/2} \left( \frac{\sin(\pi x)}{\sin(\pi t)}\right)^2 \sin(n\pi t)^2 dt \\
    & \geq 
    \frac{1}{n^2 \sin(\pi x)^2} \int_x^{1/2} \left( \frac{x}{t}\right)^2 \sin(n\pi t)^2 dt \\
    & =
    \frac{\pi x^2}{n \sin(\pi x)^2} \int_{n \pi x}^{n\pi/2} \left( \frac{\sin s}{s}\right)^2  ds.
\end{align*}
Therefore, %since $g(x) \leq \frac{1}{n^2 \sin(\pi x)^2}$, 
the inequality $g(x) \leq 5.67 n \int_x^{1/2} g(t)dt $ would be a consequence of
$$
\sin(n\pi x)^2 \leq 5.67  n^2 \pi x^2 \int_{n \pi x}^{n\pi/2} \left( \frac{\sin s}{s}\right)^2  ds .
$$
Set $y\coloneqq n\pi x$. We need to prove that
$$
\pi \sin(y)^2 \leq 5.67  y^2  \int_{y}^{n\pi/2} \left( \frac{\sin s}{s}\right)^2  ds 
$$
holds for all $y \in [\pi, \frac{n\pi}{2}-\pi]$. This would be a consequence of
$$
\pi \sin(y)^2 \leq 5.67  y^2  \int_{y}^{y+\pi} \left( \frac{\sin s}{s}\right)^2  ds , \qquad y \geq \pi .
$$
Now, observe that
$$
\int_{y}^{y+\pi} \left( \frac{\sin s}{s}\right)^2  ds \geq \frac{1}{(y+\pi)^2}\int_y^{y+\pi} \sin^2 s \; ds 
=
\frac{\pi}{2(y+\pi)^2}   .
$$
Therefore, it suffices to prove that, for any $y \geq \pi$, it holds $2  (y+\pi)^2 \sin^2 y \leq 5.67  y^2$, which is a consequence of
$(y+\pi) |\sin y| \leq 1.68 y$ proved  in Lemma \ref{lem:final} below. 

As an intermediate conclusion, we established  \eqref{eq:start} for all $n$ except $n=5$ and we are left with proving \eqref{eq:forCyril} only in the regime $[1/5, 1/2-1/5]=[1/5,3/10]$.
Our starting point is Inequality \eqref{eq:pain}. As for the case $n \geq 6$, we need to bound from below $1/g(x)$ and $\int_x^{1/2}g(t)dt$. Using similar arguments, we have for $\theta \in (0,1/2)$ and $x \in [1/5, (1+\theta)/5]$,
\begin{align*}
\frac{1}{g(x)}  = 
\frac{25\sin(\pi x)^2}{\sin(5 \pi x)^2}  
 \geq 
\frac{25 \sin(\pi/5)^2}{\sin(\theta \pi)^2} 
%\simeq 
%\frac{9}{\sin(\theta \pi)^2} 
\end{align*}
and for $x  \in [(1+\theta)/5,3/10]$, 
\begin{align*}
\frac{1}{g(x)} 
& = 
\frac{25 \sin(\pi x)^2}{\sin(5 \pi x)^2} 
\geq 
25 \sin^2\left(\frac{(1+\theta)\pi}{5} \right) 
 .
\end{align*}
Therefore,
$$
\frac{1}{g(x)}  \geq 25 \max_{\theta \in (0,1/2)} \min \left( \frac{\sin(\pi/5)^2}{\sin(\theta \pi)^2} ,  \sin^2\left(\frac{(1+\theta)\pi}{5} \right)  \right) .
$$
For $\theta = 0.299$, the above minimum equals $\frac{\sin(\pi/5)^2}{\sin(\theta \pi)^2}$ from which we deduce that $1/g(x) \geq 13.25$.

It follows that
$$
\frac{g(x)}{2\sqrt{\pi(n^2-1)\log(1/g(x))}} 
\leq 
\frac{\sin(5 \pi x)^2}{\sin(\pi x)^2}
\frac{1}{25 \times 2\sqrt{24 \pi \log(13.25)}} 
\leq 
 \frac{\sin(5 \pi x)^2}{25 \times 27.9 \sin(\pi x)^2} .
$$
On the other hand,
$$
\int_x^{1/2} g(t) dt 
\geq 
\frac{1}{25} \int_x^{x+\frac{1}{5}} \sin(5 \pi t)^2 dt
=
\frac{1}{250} .
$$
Therefore, \eqref{eq:pain} would be a consequence of
$\sin(5 \pi x)^2 \leq 2.79 \sin(\pi x)^2$, $x \in [1/5
,3/10]$. Setting $y=\pi x$ and taking the root (note that $5y \in [\pi, 3\pi/2]$), it is enough to prove
$$
-\sin(5 y) \leq 1.67 \sin y , \qquad y \in [\pi/5, 3\pi/10] .
$$
As a last reduction, we set $z=\sin(y)$ and use that $\sin(5y)=16z^5-20z^3+5z$ so that the latter can be recast as $16z^4-20z^2+6.67 \geq 0$
for all $z \in [\sin(\pi/5), \sin(3\pi/10)]$. But the second order polynomial $16X^2-20X+6.67$ is always positive since its discriminant is negative, ending the proof of the lemma provided we can prove
Lemma \ref{lem:James} and Lemma \ref{lem:final}.
\end{proof}

In the proof of Lemma \ref{lem:T'discrete} we used the following lemma.

\begin{lem} \label{lem:final}
For all $y \geq \pi$, it holds 
\begin{equation} \label{eq:final}
 (y+\pi) |\sin y| \leq 1.68 y. 
\end{equation} 
\end{lem}

\begin{proof}%[Proof fo Lemma \ref{lem:final}]
For $y \geq 3\pi/2$, we have
$$
\frac{y+\pi}{y}  \leq \frac{5}{3}  \leq 1.68 .
$$
from which \eqref{eq:final} follows. Therefore, it remains to prove \eqref{eq:final} for $\pi \leq y \leq 3\pi/2$. 
Let 
$$
H(y)\coloneqq 1.68 y - (y+\pi)|\sin y| = 1.68 y + (y+\pi)\sin y,  \qquad y \in (\pi,3\pi/2) .
$$
The  successive derivatives are
$$
H'(y) = 1.68 + \sin y + (y + \pi) \cos y, \qquad
H''(y)  = 2 \cos y - (y+\pi)\sin y 
$$
and
$$
H'''(y) = -3 \sin y -(y+\pi) \cos y .
$$
Notice that, for $y \in (\pi,3\pi/2)$, $\cos y, \sin y \leq 0$ so that $H''' > 0$ and $H''$ is increasing.
Since $H''(\pi)=-2$ and $H''(3\pi/2)=5\pi/2$, $H''$ changes sign from $-$ to $+$ at a unique point $y_o$.
Therefore $H'$ is decreasing on $(\pi,y_o)$ and increasing on $(y_o,3\pi/2)$. Since 
$H'(\pi)=-0.32$ and $H'(3\pi/2)=0.68$ we can conclude that $H'$ changes sign from $-$ to $+$ at a point $y_1 > y_o$ and that $H$ has a unique minimum at $y_1$. Therefore, the thesis will follow if we can prove that $H(y_1) \geq 0$.

Now by construction $H'(y_1)=0$. Hence
$$
1.68 \sin y_1 + \sin^2 y_1 + (y_1 + \pi) \sin y_1 \cos y_1 = 0
$$
from which we deduce that
$$
(y_1 + \pi) \sin y_1 = -\frac{1}{\cos y_1} \left( 1.68 \sin y_1 + \sin^2 y_1  \right)
$$
(note that $\pi < y_1 < 3\pi/2$ so that $\cos y_1 \neq 0$). In turn
$$
H(y_1) = \frac{1}{\cos y_1} \left( 1.68y_1 \cos y_1 - 1.68 \sin y_1 - \sin^2 y_1  \right)
$$
Since $\cos y_1 <0$, $H(y_1) \geq 0$ amounts to proving that
$G(y) \coloneqq 1.68y \cos y - 1.68 \sin y - \sin^2 y$, $y \in (\pi,3\pi/2)$, satisfies
$G(y_1) \leq 0$. 

Since $G'(y)=-\sin y(1.68 y + 2 \cos y)$, and since $1.68 y + 2 \cos y \geq 0$ on $(\pi,3\pi/2)$
(a consequence of the fact that $y \mapsto 1.68 y + 2 \cos y$ is increasing on $(\pi,3\pi/2)$), $G$ is increasing. Therefore, to prove that $G(y_1) \leq 0$, it is enough to find $y^* \geq y_1$ with $G(y^*) \leq 0$.

To that aim, take $y^*:=4.6244$. Observe that $H'(y^*) \simeq 0.0014$ so that, since $H'$  changes sign from $-$ to $+$ and $H'(y_1)=0$, necessarily $y^* \geq y_1$. Since  $G(y^*) \simeq -0.0015$ the lemma is proved.
\end{proof}

Next we prove Lemma \ref{lem:James}.

\begin{proof}[Proof of lemma \ref{lem:James}]
We distinguish  different cases.
\begin{itemize}
\item{$n \geq 6$.}
For $x \in [1/2- 1/n, 1/2]$, $g(x) = \frac{\sin^2(n \pi x)}{n^2 \sin^2(\pi x)} \leq \frac{1}{n^2 \sin^2(\pi(1/2 - 1/n))} = \frac{1}{n^2 \cos^2(\pi/n)}$.  Thus by Lemma \ref{lem: reform in terms of erfc} {below} it suffices to prove
\begin{align*}
    \frac{1}{6} \frac 1 {n^2 \cos^2(\pi/n)} + \frac{1}{2} \left(\frac 1 { n^2 \cos^2(\pi /n)}\right)^{\frac 4 3} \leq \frac{1}{2n^2},
\end{align*}
or, rearranging,
$$
 \frac{1}{6} + \frac{1}{2} \frac 1 { \left( n \cos(\pi /n)\right)^{\frac 2 3}}
\leq \frac{1}{2} \cos^2(\pi/n) .
$$
Note that the left hand side is decreasing in $n$, while the right hand side is increasing in $n$, so to prove the result for all $n \geq 6$, one needs only check the $n = 6$ case, which can be evaluated exactly,
$$
\frac{1}{6} + \frac{1}{2} \frac 1 { \left( 6 \cos(\pi /6)\right)^{\frac 2 3}} = \frac{1}{6} + \frac{1}{6} = \frac{1}{3}\leq \frac{3}{8} = \frac{1}{2} \cos^2(\pi/6) .
$$

%\item $n=2$
    
 %   When $n =2$, $(1/2,1/2] \cap [1/2- 1/2, 1/2] = \emptyset$.\\
    
    \item{$n=3$.}
    When $n=3$, a derivative computation shows that the maximum of $g(x)$ on $[1/3,1/2]$ occurs when $x = 1/2$, with $g(1/2) = \frac 1 9$.  Thus to finish this case we need only prove $f(A) \leq \frac 1 9$ or by Lemma \ref{lem: reform in terms of erfc}
    \begin{align*}
        \frac{1}{6 \times 9} + \frac{1}{2 \times 9^{4/3}} { < } 0.046 { < }  0.55 { < } \frac{1}{2 \times 3^2} .
    \end{align*}

    \item{$n =4$.}
    In this case $g'(x) = - \frac 1 2 \pi \sin(4 \pi x) \left( 3 \cos(2 \pi x) + 2 \right)$ shows that $g$ takes its maximum on $[1/4,1/2]$ at $x_0 = { {\arccos(-2/3)}/{(2\pi)}}$.  Applying half and double angle formulas for $\sin$ and $\sin(\arccos(x)) = \sqrt{ 1- x^2}$ yields $g(x_0) = \frac 2 {27}.$  Similarly to prove that $f(A) \leq g(x)$ on $(1/4,1/2)$ by Lemma \ref{lem: reform in terms of erfc} we need to check
    \begin{align*}
        \frac{2}{6 \times 27} + \frac{2^{4/3}}{2 \times 27^{4/3}}  \leq \frac{1}{2 \times 4^2} .
    \end{align*}
    This can be checked by hand or by numerical approximation where $\frac{2}{6 \times 27} + \frac{2^{4/3}}{2 \times 27^{4/3}} { < }  0.028$ and $\frac{1}{2 \times 4^2} { = 0.03125}$.
    
\item{$n=5$.}
In this case, the maximum of $g$ on $[2/5,1/2]$ on occurs at $x_0 = 1/2$, where $g(1/2) = 1/25$.  On $[1/5,2/5]$, $g$ takes its maximum at $x_0 = \frac 2 \pi \arctan\left( \sqrt{ \frac{11-4\sqrt{6}}{5}}\right)$, with $g(x_0) = \frac 1 {16}$. Thus  by Lemma \ref{lem: reform in terms of erfc} we need only check 
$$
\frac{1}{\sqrt{\pi} 16 \sqrt{\log(16)}} \left( 1 - \frac{1}{2 \log(16)} + \frac{3}{4\log(16)^2}\right) \leq \frac{1}{2 \times 5^2}
$$
The result follows since the left hand side approximately equals $0.019$
and the right hand side equals $0.02$.
\end{itemize}
This achieves the proof.
\end{proof}

\begin{lem} \label{lem: reform in terms of erfc}
Given $I \subset [0,1/2]$, set $g^* \coloneqq \sup_{x \in I} g(x)$. Then,
to prove $f(A) \geq g(x)$ on $I$, it suffices to prove
\begin{align} \label{eq:g*}
    \frac{ g^*}{6} + \frac{g^{*4/3}}{2}
    \leq \frac{1}{2n^2} 
\end{align}
provided $g^* \leq e^{-1/4}$, or to prove
\begin{equation} \label{eq:g*2}
\frac{g^*}{\sqrt{\pi} \sqrt{\log(1/g^*)}} \left( 1 - \frac{1}{2 \log(1/g^*)} + \frac{3}{4\log(1/g^*)^2}\right) \leq \frac{1}{2n^2}  .
\end{equation}
\end{lem}

\begin{proof}
The definition of $A$,
$
    \frac 1 {2n} 
        = 
            \int_0^A e^{- \pi (n^2-1) x^2} dx
$, can be written through change of variables $ y = \sqrt{ \pi (n^2-1)} x$ as 
\begin{align*}
    \Erf( \sqrt{\pi (n^2-1)} A ) = \sqrt{ 1 - \frac 1 {n^2}}
\end{align*}
where $\Erf(x) \coloneqq \frac{2}{\sqrt{\pi}} \int_0^x e^{-t^2} dt$ is the error function.  We denote its inverse function $\Erf^{-1}$ and can give an explicit expression for $A$,
\begin{align*}
    A = \frac{  \Erf^{-1}\left( \sqrt{ 1 - \frac 1 {n^2}}\right)}{\sqrt{ \pi (n^2-1)}} .
\end{align*}
Thus,  the relation
\begin{align*}
    f(A) = e^{-\pi (n^2 -1) A^2} = e^{- \Erf^{-1}\left( \sqrt{1 - \frac 1 {n^2}}\right)^2} \geq g(x), 
\end{align*}
can be rearranged as
\begin{align*}
      \Erfc\left( \sqrt{ \log \frac 1 {g(x)} } \right) \leq 1 - \sqrt{ 1 - \frac 1 {n^2}} .
\end{align*}
Since 
$$
\frac{1}{2n^2} \leq 1 - \sqrt{ 1 - \frac 1 {n^2}}  
$$
it is  enough to prove that
$$
\Erfc\left( \sqrt{ \log \frac 1 {g(x)} } \right) \leq \frac{1}{2n^2} .
$$
The first claim follows from the following  upper bound for the Erfc function {which} is popular in Engineering (see \textit{e.g.}\ \cite{chiani2003new}) valid for  $x \geq 1/2$, 
\begin{align*} %\label{eq: sharp EE erfc bound}
    \Erfc(x) \leq \frac{e^{-x^2}}{6} + \frac{e^{-4x^2/3}}{2}.
\end{align*}
For the second claim, we use successive integration by parts to get that, for $x>0$,
\begin{align*}
\frac{\sqrt{\pi}}{2} \Erfc\left( x \right) 
& = 
\int_x^\infty e^{-t^2}dt 
=    
\int_x^\infty \frac{-2t e^{-t^2}}{-2t}dt
=
\frac{e^{-x^2}}{2x} - \frac{1}{2}\int_x^\infty \frac{e^{-t^2}}{t^2}dt \\
& =
\frac{e^{-x^2}}{2x} - \frac{1}{4} \frac{e^{-x^2}}{x^3} + \frac{3}{4}\int_x^\infty \frac{e^{-t^2}}{t^4}dt \\
& =
\frac{e^{-x^2}}{2x} - \frac{1}{4} \frac{e^{-x^2}}{x^3}
+ \frac{3}{8} \frac{e^{-x^2}}{x^5} - 
\frac{15}{8}\int_x^\infty \frac{e^{-t^2}}{t^6}dt \\
& \leq
\frac{e^{-x^2}}{2x} - \frac{1}{4} \frac{e^{-x^2}}{x^3}
+ \frac{3}{8} \frac{e^{-x^2}}{x^5}
\end{align*}
from which the second claim follows.
\end{proof}

%%%%%%%%%%%%%%%%%%%%
%%%%%%%%%%%%%%%%%%%%  Biblio

\bibliographystyle{plain}
\bibliography{Nazarov-distribution}

\begin{thebibliography}{10}

\bibitem{astashkin2021majorization}
S.~V. Astashkin, K.~V. Lykov, and M.~Milman.
\newblock Majorization revisited: Comparison of norms in interpolation scales.
\newblock {\em arXiv preprint arXiv:2107.11854}, 2021.

\bibitem{ball}
K.~Ball.
\newblock Cube slicing in {${\bf R}^n$}.
\newblock {\em Proc. Amer. Math. Soc.}, 97(3):465--473, 1986.

\bibitem{brenier}
Y.~Brenier.
\newblock Polar factorization and monotone rearrangement of vector-valued
  functions.
\newblock {\em Communications on pure and applied mathematics}, 44(4):375--417,
  1991.

\bibitem{brzezinski2013volume}
P.~Brzezinski.
\newblock Volume estimates for sections of certain convex bodies.
\newblock {\em Mathematische Nachrichten}, 286(17-18):1726--1743, 2013.

\bibitem{caffarelli2000monotonicity}
L.A. Caffarelli.
\newblock Monotonicity properties of optimal transportation and the {FKG} and
  related inequalities.
\newblock {\em Communications in Mathematical Physics}, 214(3):547--563, 2000.

\bibitem{caffarelli2002monotonicity}
L.A. Caffarelli.
\newblock Erratum: ``{M}onotonicity of optimal transportation and the {FKG} and
  related inequalities'' [{C}omm. {M}ath. {P}hys. {\bf 214} (2000), no. 3,
  547--563; {MR}1800860 (2002c:60029)].
\newblock {\em Comm. Math. Phys.}, 225(2):449--450, 2002.

\bibitem{chiani2003new}
M.~Chiani, D.~Dardari, and M.~K. Simon.
\newblock New exponential bounds and approximations for the computation of
  error probability in fading channels.
\newblock {\em IEEE Transactions on Wireless Communications}, 2(4):840--845,
  2003.

\bibitem{chong}
K.M. Chong.
\newblock Some extensions of a theorem of {H}ardy, {L}ittlewood and {P}\'{o}lya
  and their applications.
\newblock {\em Canadian J. Math.}, 26:1321--1340, 1974.

\bibitem{diamond}
P.A. Diamond and J.E. Stiglitz.
\newblock Increases in risk and in risk aversion.
\newblock {\em J. Econom. Theory}, 8(3):337--360, 1974.

\bibitem{dirksen}
H.~Dirksen.
\newblock Hyperplane sections of cylinders.
\newblock {\em Colloq. Math.}, 147(1):145--164, 2017.

\bibitem{fathi2020aproof}
M.~Fathi, N.~Gozlan, and M.~Prod'homme.
\newblock A proof of the {C}affarelli contraction theorem via entropic
  regularization.
\newblock {\em Calc. Var. Partial Differential Equations}, 59(3):Paper No. 96,
  18, 2020.

\bibitem{gozlan-juillet}
N.~Gozlan and N.~Juillet.
\newblock On a mixture of {B}renier and {S}trassen theorems.
\newblock {\em Proc. Lond. Math. Soc. (3)}, 120(3):434--463, 2020.

\bibitem{haagerup}
U.~Haagerup.
\newblock The best constants in the {K}hintchine inequality.
\newblock {\em Studia Math.}, 70(3):231--283 (1982), 1981.

\bibitem{harge}
G.~Harg\'{e}.
\newblock A convex/log-concave correlation inequality for {G}aussian measure
  and an application to abstract {W}iener spaces.
\newblock {\em Probab. Theory Related Fields}, 130(3):415--440, 2004.

\bibitem{karlin}
S.~Karlin.
\newblock {\em Total positivity. {V}ol. {I}}.
\newblock Stanford University Press, Stanford, Calif, 1968.

\bibitem{karlin-novikoff}
S.~Karlin and A.~Novikoff.
\newblock Generalized convex inequalities.
\newblock {\em Pacific J. Math.}, 13:1251--1279, 1963.

\bibitem{konig-koldobsky}
H.~K\"{o}nig and A.~Koldobsky.
\newblock On the maximal perimeter of sections of the cube.
\newblock {\em Adv. Math.}, 346:773--804, 2019.

\bibitem{MMR20}
M.~Madiman, J.~Melbourne, and C.~Roberto.
\newblock Bernoulli sums and {R}\'enyi entropy inequalities.
\newblock {\em arXiv preprint arXiv:2103.00896}, 2021.

\bibitem{MMX2017unpublished}
M.~Madiman, M.~Melbourne, and P.~Xu.
\newblock Unpublished notes.
\newblock Preprint, 2017.

\bibitem{marshall}
A.~W. Marshall, I.~Olkin, and B.~C. Arnold.
\newblock {\em Inequalities: theory of majorization and its applications}.
\newblock Springer Series in Statistics. Springer, New York, second edition,
  2011.

\bibitem{MOP}
Albert~W Marshall, Ingram Olkin, and Frank Proschan.
\newblock Monotonicity of ratios of means and other applications of
  majorization.
\newblock Technical report, Boeing Scientific Research Labs Seattle, WA,
  Mathematics Research Lab, 1965.

\bibitem{MP}
V.~D. Milman and A.~Pajor.
\newblock Isotropic position and inertia ellipsoids and zonoids of the unit
  ball of a normed n-dimensional space.
\newblock {\em Geometric aspects of functional analysis}, pages 64--104, 1989.

\bibitem{Mordorst}
O.~Mordhorst.
\newblock The optimal constants in {K}hintchine's inequality for the case
  {$2<p<3$}.
\newblock {\em Colloq. Math.}, 147(2):203--216, 2017.

\bibitem{NP00}
F.~L. Nazarov and A.~N. Podkorytov.
\newblock Ball, {H}aagerup, and distribution functions.
\newblock In {\em Complex analysis, operators, and related topics}, volume 113
  of {\em Oper. Theory Adv. Appl.}, pages 247--267. Birkh\"auser, Basel, 2000.

\bibitem{OP00}
K.~Oleszkiewicz and A.~Pe{\l}czy\'{n}ski.
\newblock Polydisc slicing in {${\bf C}^n$}.
\newblock {\em Studia Math.}, 142(3):281--294, 2000.

\bibitem{rao-ren}
M.~M. Rao and Z.~D. Ren.
\newblock {\em Theory of {O}rlicz spaces}.
\newblock Marcel {D}ekker {I}nc., 1991.

\bibitem{renyi1961measures}
A.~R{\'e}nyi et~al.
\newblock On measures of entropy and information.
\newblock In {\em Proceedings of the Fourth Berkeley Symposium on Mathematical
  Statistics and Probability, Volume 1: Contributions to the Theory of
  Statistics}. The Regents of the University of California, 1961.

\bibitem{shaked2007stochastic}
M.~Shaked and G.~J. Shanthikumar.
\newblock {\em Stochastic orders}.
\newblock Springer, 2007.

\bibitem{tsallis1988possible}
C.~Tsallis.
\newblock Possible generalization of {B}oltzmann-{G}ibbs statistics.
\newblock {\em Journal of statistical physics}, 52(1):479--487, 1988.

\bibitem{van2014renyi}
T.~Van~Erven and P.~Harremos.
\newblock R{\'e}nyi divergence and {K}ullback-{L}eibler divergence.
\newblock {\em IEEE Transactions on Information Theory}, 60(7):3797--3820,
  2014.

\bibitem{villani03}
C.~Villani.
\newblock {\em Topics in optimal transportation}, volume~58 of {\em Graduate
  Studies in Mathematics}.
\newblock American Mathematical Society, Providence, RI, 2003.

\bibitem{villani2009optimal}
C.~Villani.
\newblock {\em Optimal transport}, volume 338 of {\em Grundlehren der
  Mathematischen Wissenschaften [Fundamental Principles of Mathematical
  Sciences]}.
\newblock Springer-Verlag, Berlin, 2009.
\newblock Old and new.

\bibitem{watson}
G.~N. Watson.
\newblock {\em A {T}reatise on the {T}heory of {B}essel {F}unctions}.
\newblock Cambridge University Press, Cambridge, England; The Macmillan
  Company, New York, 1944.

\end{thebibliography}
  
\end{document}